\documentclass[12pt]{amsart}
\usepackage{amsmath, amsthm, amssymb,cite}
\usepackage{fullpage}
\usepackage{color}
\usepackage[colorlinks=true]{hyperref}

\newtheorem{theorem}{Theorem}[section]
\newtheorem{lemma}{Lemma}
\newtheorem{proposition}[lemma]{Proposition}

\newtheorem{definition}[lemma]{Definition}

\numberwithin{lemma}{section}

\allowdisplaybreaks

\numberwithin{equation}{section}

\newcommand{\R}{{\mathbb R}}

\renewcommand{\H}{{\mathcal H }}
\newcommand{\nP}{{\mathbf P}}
\newcommand{\CalA}{{\mathcal A}}
\newcommand{\CalB}{{\mathcal B}}

\newcommand{\ua}{{\underline a}}

\newcommand{\ub}{{\underline b}}

\newcommand{\W}{{\mathbf W}}

\author{Lizhe Wan}
\address{Department of Mathematics, University of Wisconsin - Madison}
\curraddr{}
\email{lwan33@wisc.edu}

\begin{document}

\title{On the capillary water waves with constant vorticity}

\begin{abstract}
This article is devoted to the study of local well-posedness for deep  water waves with constant vorticity in two space dimensions on the real line.
The water waves can be paralinearized and written as a quasilinear dispersive system of equations. 
By using the energy estimate and the Strichartz estimate, we show that for $s> \frac{5}{4}$, the gravity-capillary water wave system with constant vorticity is locally well-posed in $\mathcal{H}^{s}(\mathbb{R})$.
\end{abstract}

\maketitle

\tableofcontents

\section{Introduction}
In this article, we consider the two-dimensional inviscid incompressible gravity-capillary water wave system with constant vorticity in infinite depth. 
The water waves can be described using the incompressible Euler's equations with free boundary conditions on the fluid surface. 

We first briefly mention the free boundary Euler equations and the water waves with constant vorticity $\gamma$ in the Zakharov-Craig-Sulem formulation. 
Let the fluid domain at time $t$ be denoted by $\Omega_t$, and the boundary at time $t$ by $\Gamma_t$. 
The boundary $\Gamma_t$ in our setting is  either asymptotically flat at infinity or periodic with zero mean. 
The fluid velocity is denoted by $\mathbf{u}(t,x,y) = (u(t,x,y), v(t,x,y))$ and the pressure is denoted by $p(t,x,y)$. 
Then $(\mathbf{u}, p)$ solve the following incompressible Euler equations inside $\Omega_t$, 
\begin{equation*}
\left\{
             \begin{array}{lr}
            u_t +uu_x +vu_y = -p_x &  \\
            v_t + uv_x +vv_y = -p_y -g& \\
            u_x +v_y =0 & \\
            \omega = v_x - u_y = \gamma.
             \end{array}
\right.
\end{equation*}
On the  free upper boundary $\Gamma_t$ we have the kinematic boundary condition 
\begin{equation*}
\partial_t+ \mathbf{u} \cdot \nabla \text{ is tangent to } \bigcup \Gamma_t,
\end{equation*}
and also the dynamic boundary condition
\begin{equation*}
 p = -2 \sigma {\bf H}  \ \ \text{ on } \Gamma_t.
\end{equation*}
Here ${\bf H}$ is the mean curvature of the boundary, $g \geq 0$ is the gravitational constant and $\sigma>0$ is the coefficient of the surface tension.

To rewrite the free boundary Euler system, one can use  $\mathbf{u} = (\gamma y + \phi_x, \phi_y)$ for some generalized velocity potential $\phi(t, x, y)$ that is harmonic with appropriate decay at infinity or periodic with zero mean.
Then at each time $t$, $\phi(t,x,y)$ solves the Laplace equation, and $\phi$ inside $\Omega_t$ is determined by its trace on the free boundary $\Gamma_t$. 
Following the work of Zakharov \cite{zakharov1968stability} and Craig-Sulem \cite{MR1158383}, we denote by $\eta$ the height of the water surface as a function of time and the horizontal coordinate, and $\psi (t,x)=\phi (t,x,\eta(t,x))\in \mathbb{R}$ to be the trace of the generalized velocity potential $\phi$ on the boundary. 
The fluid dynamics can be expressed in terms of a one-dimensional evolution of  the pair of unknowns
$(\eta,\psi)$:
\begin{equation*}
\left\{
\begin{aligned}
& \partial_t \eta - G(\eta) \psi = \gamma \eta \eta_x \\
& \partial_t \psi +g\eta  +\frac12 \psi_x^2 - \frac12 
\frac{( \eta_x   \psi_x +G(\eta) \psi)^2}{1 + \eta_x^2} = \sigma \left( \frac{\eta_x}{\sqrt{1+\eta_x^2}}\right)_x + \gamma \eta \psi_x + \gamma \partial_x^{-1}G(\eta) \psi,
\end{aligned}
\right.
\end{equation*}
where $G(\eta)$ is the Dirichlet-Neumann operator
\begin{equation*}
    G(\eta)\psi : = (-\phi_x \eta_x + \phi_y )|_{y = \eta(x)}.
\end{equation*}
The first equation of the system  is the equation for the kinematic boundary condition, and the second equation is the equation for the dynamic boundary condition.
This is the Zakharov-Craig-Sulem formulation of the   water waves.

\subsection{Water waves in holomorphic coordinates}
In this article, we will not use the Zakharov-Craig-Sulem formulation, instead, we will work in the holomorphic coordinates introduced by Hunter-Ifrim-Tataru \cite{MR3535894}.
The holomorphic coordinates are widely used to study a variety of water waves problems, see for example Ifrim-Tataru and Ai-Ifrim-Tataru \cite{MR3535894, MR3499085, MR3625189,  ai2023dimensional, MR4462478, MR4483135, MR3667289, MR3869381}.

As seen in the water wave system below, we no longer need to analyze the Dirichlet-Neumann operator in the holomorphic coordinates.
Let $\mathbf{P}$ be the projection onto negative frequencies, namely
\begin{equation*}
    \mathbf{P} := \frac{1}{2}(\mathbf{I} - iH),
\end{equation*}
with $H$ being the Hilbert transform.
Holomorphic functions are defined to be the functions whose Fourier transforms are supported on $(-\infty,0]$.
In other words, they satisfy the relation $\nP f = f$.
Similarly, we define $\bar{\nP}$ to be the projection onto positive frequencies:
\begin{equation*}
    \bar{\nP} := \frac{1}{2}(\mathbf{I} + iH) = \mathbf{I} -\nP.
\end{equation*}
Functions such that $\bar{\nP} f = f$ are called anti-holomorphic functions. 
Anti-holomorphic functions are complex conjugates of holomorphic functions.

The water waves can be formulated in the class of holomorphic functions using the pair of unknowns $(W(t,\alpha), Q(t,\alpha))$.
 $W$ represents the holomorphic position and  $Q$ represents the holomorphic velocity potential.
In the rest of this article, we will use $\alpha$ to denote the holomorphic spacial variable.
The derivation of the capillary water wave equations in holomorphic coordinates can be found in Ifrim-Tataru \cite{MR3667289} and Appendix B of Ifrim-Tataru \cite{MR3869381}. 
The system is given by
\begin{equation}
\left\{
             \begin{array}{lr}
             W_t + (W_\alpha +1)\underline{F} +i\dfrac{\gamma}{2}W = 0 &  \\
             Q_t - igW +\underline{F}Q_\alpha +i\gamma Q +\mathbf{P}\left[\dfrac{|Q_\alpha|^2}{J}\right]- i\dfrac{\gamma}{2}T_1-2\sigma \mathbf{P}\Im\left[ \dfrac{W_{\alpha \alpha}}{J^{\frac 1 2}(1+W_{\alpha})}\right]  =0,&  
             \end{array}
\right.\label{e:CVWW}
\end{equation}
where $J := |1+ W_\alpha|^2$ is the Jacobian, and
\begin{equation*}
\begin{aligned}
&F: = \mathbf{P}\left[\frac{Q_\alpha - \Bar{Q}_\alpha}{J}\right], \quad &F_1 = \mathbf{P}\left[\frac{W}{1+\Bar{W}_\alpha}+\frac{\Bar{W}}{1+W_\alpha}\right],\\
&\underline{F}: =F- i \frac{\gamma}{2}F_1,  \quad &T_1: = \mathbf{P}\left[\frac{W\Bar{Q}_\alpha}{1+\Bar{W}_\alpha}-\frac{\Bar{W}Q_\alpha}{1+W_\alpha}\right].
\end{aligned}
\end{equation*}

The system \eqref{e:CVWW} has a conserved  energy
\begin{equation*}
\mathcal{E}(W, Q)= \Re\int Q\bar{Q}_{\alpha}
+2\sigma (J^{\frac{1}{2}} - 1 -\Re W_\alpha) + g|W|^2(1+W_\alpha) + \gamma Q_\alpha (\Im W)^2 - \frac{\gamma^3}{2i}|W|^2W(1+W_\alpha) \, d\alpha. 
\end{equation*}
It also has a conserved horizontal momentum
\begin{equation*}
    \mathcal{P}(W,Q) = \int -i(Q\bar{W}_\alpha-\bar{Q}W_\alpha) -\gamma |W|^2 + \frac{\gamma}{2}(W^2 \bar{W}_\alpha + \bar{W}^2 W_\alpha) \,d\alpha.
\end{equation*}
A simplified model of \eqref{e:CVWW} is its linearization around the zero solution
\begin{equation}
\left\{
             \begin{array}{lr}
             w_t + q_\alpha = 0 &  \\
             q_t+i\gamma q -igw + i\sigma w_{\alpha\alpha} =0,&  
             \end{array}
\right.\label{e:ZeroLinear}
\end{equation}
restricted to holomorphic functions.
By eliminating the linearized unknown $q$, \eqref{e:ZeroLinear}  can be written as a linear dispersive equation pf $w$:
\begin{equation*}
    w_{tt}+ i\gamma w_t + igw_\alpha - i\sigma w_{\alpha\alpha \alpha} =0.
\end{equation*}
Its dispersion relation is 
\begin{equation*}
    \tau^2 + \gamma \tau +g\xi + \sigma\xi^3 =0, \quad \xi\leq 0.
\end{equation*}

A conserved energy of the linear system \eqref{e:ZeroLinear} is given by
\begin{equation*}
    \mathcal{E}_0(w,q) = \int \sigma|w_\alpha|^2 - \frac{i}{2}(q\bar{q}_\alpha - \bar{q}q_\alpha)\, d\alpha = \sigma\|w\|_{\dot{H}^1}^2 + \|q\|_{\dot{H}^{\frac{1}{2}}}^2. 
\end{equation*}
This conserved energy suggests the functional framework to study \eqref{e:CVWW}.
The system \eqref{e:ZeroLinear} is well-posed in $\dot{H}^1 \times \dot{H}^{\frac{1}{2}}$ space.
To measure the regularity of the solution we will use the product Sobolev space $\mathcal{H}^{s}$ endowed with the norm
\begin{equation*}
    \| (w,q)\|_{\mathcal{H}^{s}} := \|(w,q) \|_{ H^{s+\frac{1}{2}}\times H^s}, \quad s\in \mathbb{R}.
\end{equation*}
We will also use the product H\"{o}lder space $\mathcal{W}^r$ endowed with the norm
\begin{equation*}
\|(w,q) \|_{\mathcal{W}^r} := \|(w,q) \|_{W^{r+\frac{1}{2},\infty}\times W^{r,\infty}}, \quad r\in \mathbb{R}.
\end{equation*}

The  system \eqref{e:CVWW} is a nonlinear degenerate hyperbolic system.
By differentiation, it can be diagonalized and converted into a quasilinear system.
Following the setup in \cite{MR3667289}, we work with holomorphic functions
\begin{equation*}
(\W, R): = \left(W_\alpha, \frac{Q_\alpha}{1+W_\alpha}\right).
\end{equation*}
The holomorphic function $R$ above has an intrinsic meaning; it is the complex velocity restricted on the water surface.
We also need to define three auxiliary functions.
The first one $\ua$ is the \textit{frequency-shift}, and it is given by
 \begin{equation*}
     \ua := a + \frac{\gamma}{2}a_1, \quad a := i (\bar{\nP}[\bar{R}R_\alpha] - \nP[R\bar{R}_\alpha]), \quad a_1 := R+ \bar{R} -N,
 \end{equation*}
 where 
 \begin{equation*}
     N: = \nP[W\bar{R}_\alpha - \bar{\W}R]+ \bar{\nP}[\bar{W}R_\alpha - \W\bar{R}].
 \end{equation*}
The second auxiliary function we need is the \textit{advection velocity} $\ub$,
\begin{equation*}
    \ub := b-i\frac{\gamma}{2} b_1, \quad b: = \nP \left[\frac{Q_\alpha}{J}\right] + \bar{\nP} \left[\frac{\bar{Q}_\alpha}{J} \right], \quad b_1: = \nP \left[ \frac{W}{1+\bar{\W}}\right]- \bar{\nP} \left[ \frac{\bar{W}}{1+\W}\right].
\end{equation*}
The third  auxiliary function $\underline{M}$ is given by
\begin{align*}
    & M := \frac{R_\alpha}{1+\bar{\W}}+ 
    \frac{\bar{R}_\alpha}{1+\W} - b_\alpha = \bar{\nP}[\bar{R}Y_\alpha - R_\alpha \bar{Y}] + \nP[R\bar{Y}_\alpha -\bar{R}_\alpha Y], \\
  & M_1 := \W -\bar{\W} - b_{1,\alpha} = \nP[W\bar{Y}]_\alpha - \bar{\nP}[\bar{W}Y]_\alpha, \qquad \underline{M} := M- i \frac{\gamma}{2}M_1.
\end{align*}
Note that $\ub_\alpha$ satisfies the relation
\begin{equation}
    \ub_\alpha = \frac{R_\alpha}{1+\bar{\W}}+ \frac{\bar{R}_\alpha}{1+\W} -i\frac{\gamma}{2}(\W -\bar{\W}) -\underline{M}. \label{ubalpha}
\end{equation}
Here, $Y:=\dfrac{W_{\alpha}}{1+W_{\alpha}}$. 
Then, using above notations, the pair $(\W,R)$ solves the differentiated system
\begin{equation} \label{e:WR}
\left\{
\begin{aligned}
 & \mathbf{W}_t +\underline{b}\mathbf{W}_\alpha + \dfrac{(1+\mathbf{W})R_\alpha}{1+\Bar{\mathbf{W}}} = (1+\mathbf{W})\underline{M}+i\dfrac{\gamma}{2}\mathbf{W}(\mathbf{W}-\Bar{\mathbf{W}}) \\
& R_t + \underline{b}R_\alpha +i\gamma R - i\dfrac{g\mathbf{W}-a}{1+\mathbf{W}} - \frac{ 2\sigma}{1+\W}\nP\Im \left[ \frac{\W_{ \alpha}}{J^{1/2}(1+\W)}\right]_{\alpha}  =  i\dfrac{\gamma}{2}\dfrac{R\mathbf{W}+\Bar{R}\mathbf{W}
             +N}{1+\mathbf{W}}.
\end{aligned}
\right.
\end{equation}
In the following, we will mostly work with the system \eqref{e:WR} and study its local well-posedness.

\subsection{Some previous results}
The literature on the two-dimensional gravity-capillary water waves is numerous. 
Here we only mention some of the results on the local well-posedness.

For the irrotational gravity-capillary water waves, Nalimov \cite{MR0609882} and Yosihara \cite{MR0660822, MR0728155} initiated  the study of this problem for small smooth data. 
There are  many results since then, see for example the work of Agrawal \cite{MR4244258, MR4684336}, Ambrose-Masmoudi \cite{MR2162781}, Beyer-Gunther \cite{MR1637554},  Coutand-Shkoller \cite{MR2291920}, Lannes \cite{MR3060183}, Ming-Zhang \cite{MR2558419}, M\'{e}sognon-Gireau\cite{MR3592680} and Shatah-Zeng \cite{MR2763036}.
Let us mention that Alazard-Burq-Zuily \cite{MR2805065, MR2931520} worked in Zakharov-Craig-Sulem formulation and used paradifferential calculus to reduce the water wave system to a single paradifferential equation.
As a result, they proved that 2D irrotational gravity-capillary water wave system is locally well-posed  in $\H^s$, for $s> \frac{3}{2}$.
De Poyferr\'{e} and Nguyen refined the paradifferential reduction in \cite{MR3770970}, and proved a Strichartz estimate \cite{MR3487264} that leads to the local well-posedness in $\H^{s}(\mathbb{R})$ for $s>\frac{27}{20}$. 
Nguyen in \cite{MR3724757} then used the paracomposition and the Strichartz estimate to show the local well-posedness for  the 2D irrotational gravity-capillary water waves in $\H^{s}(\mathbb{R})$ for $s>\frac{5}{4}$. 
See also the  work of Ai \cite{ai2023improved} for the Strichartz estimate and its application in the local well-posedness.

For the 2D gravity water wave problem posed in infinite depth and with constant vorticity, the first   local well-posedness for large data, as well as cubic lifespan bounds for small data solutions were obtained by Ifrim and Tataru in \cite{MR3869381}.
Shortly after, Bieri, Miao, Shahshahani and  Wu \cite{bieri2015motion} obtained a similar result for self-gravitating  incompressible fluids with
constant vorticity for smooth initial data in bounded domains.

For the gravity-capillary water waves with non-trivial constant vorticity, to the best of our knowledge, the only well-posedness result is the recent work of \cite{MR4658635} by Berti-Maspero-Murgante.
They showed almost global-in-time existence of small amplitude solutions of the 2D gravity-capillary water wave equations with constant vorticity in the periodic case.
However, their result requires high Sobolev regularity for initial data, and the parameter $\sigma$ needs to avoid a measure zero set to ensure a non-resonance condition.
Apart from the general Cauchy problem, several types of special solutions of \eqref{e:CVWW} are proved to exist, see the work of Groves-Wahl\'{e}n\cite{MR3415532}, Rowan-Wan \cite{rowan2023dimensional, rowan2024} for solitary waves, Wahl\'en \cite{MR2262949}, Martin \cite{MR2969824} for the steady periodic waves and Berti-Franzoi-Maspero  \cite{MR4228858} for quasi-periodic travelling waves solution.
No well-posedness result of \eqref{e:WR} on the real line is known before.

Lastly, the reference list of works pertaining to gravity water waves in 2D with trivial vorticity is exhaustive and by no means we intend is to cover it here. However, in this current work, we also get some inspiration from the very recent result of Ai, Ifrim and Tataru \cite{ai2023dimensional}; here we are mostly inspired by the pointwise bounds derived for the water waves associated quantities.
\subsection{The main results} 
We first define two control norms that will be used.
Let $\epsilon>0$ be a small positive constant.
\begin{align*}
&\CalA : = \|\W \|_{ C^{1+\epsilon}_{*}}+\|R\|_{C_{*}^{\frac{1}{2}}} + \gamma \| \W\|_{ C^{\frac{1}{2}}_{*}}, \\
&\CalB : =  \|\W \|_{C_{*}^{\frac{3}{2}}} + \|R \|_{C_{*}^{1+\epsilon}} +  \gamma \|\W \|_{C_{*}^{1+\epsilon}}+ \gamma \|R\|_{C_{*}^{\frac{1}{2}}}.
\end{align*}
Here the vorticity $\gamma$ can be intuitively viewed as half a derivative as in the gravity water waves \cite{MR3869381, wan2023low}.

The first main result is the a priori energy estimate for the water wave system \eqref{e:WR}.
\begin{theorem} \label{t:EnergyEstimate}
Let $s> 0$.
Suppose $(\W, R)$ solve the water wave system \eqref{e:WR} on $[0, T]$, then we have the energy estimate for any $t\in [0,T]$, 
\begin{equation}
    \frac{d}{dt}\|(\W, R) \|_{\mathcal{H}^s} \lesssim_\CalA (1+\CalB ) \|(\W, R) \|_{\mathcal{H}^s}. \label{EstimateWREnergy}
\end{equation} 
\end{theorem}
The proof of this energy estimate does not rely on the dispersive properties of the equations.
We remark that by using the energy estimate, the contraction estimate, and the Sobolev embedding for the control norms, one can show that in the periodic case, the water wave system is locally well-posed in $\mathcal{H}^s(\mathbb{T})$ for $s>\frac{3}{2}$.
Compared to the  result in \cite{MR4658635} (where the authors are using the Zakharov-Craig-Sulem formulation), this is a low regularity well-posedness result and holds for all choices of parameters $\sigma>0, \gamma\in \mathbb{R}$. 

Note that the Sobolev embedding from $H^s(\mathbb{R})$ to $C_*^{s-\frac{1}{2}}(\mathbb{R})$  exhibits a gap of $\frac{1}{2}$ derivative.
On the real line, by exploiting the dispersive properties of the water wave system, one may hope to lower this gap.
Indeed, by using the Strichartz estimate proved in Nguyen \cite{MR3724757}, one can lower the gap from $\frac{1}{2}$ to $\frac{1}{4}$.
Our second main result is the Strichartz estimate for \eqref{e:WR}.
\begin{theorem} \label{t:MainTwo}
Suppose $(\W, R)$ is a solution of \eqref{e:WR} on $I=[0,T]$ with
\begin{equation*}
(\W, R) \in C^0(I; \mathcal{H}^s(\mathbb{R}))\cap L^4(I; \mathcal{W}^r(\mathbb{R})), \quad s>r>1.
\end{equation*}
Then for $\mu<\frac{1}{4}$, $(\W, R) \in  L^4(I; \mathcal{W}^{s-\frac{1}{2}+\mu,\infty}(\mathbb{R}))$.   
\end{theorem}

Using the energy estimate and the Strichartz estimate as the main ingredients, we can prove the local well-posedness for the water wave system \eqref{e:WR} on the real line.
\begin{theorem} \label{t:MainWellPosed}
Let $s>\frac{5}{4}$, the water wave system \eqref{e:WR} is locally well-posed in $\mathcal{H}^s(\mathbb{R})$.
\end{theorem}
Here by local well-posedness in $\mathcal{H}^s(\mathbb{R})$, we need to prove
\begin{enumerate}
\item Existence of the solution: If the initial data $(\W_0, R_0)\in \mathcal{H}^s(\mathbb{R})$, then there exists a time $T>0$ such that \eqref{e:WR} has a  solution $(\W,R)\in C^0([0,T]; \mathcal{H}^s(\mathbb{R}))$. 
\item Continuous dependence of initial data: Consider a sequence of  data $\{(\W_{n,0}, R_{n,0} )\}_{n=0}^\infty$ that converges to $(\W_0, R_0)$ in $\mathcal{H}^s(\mathbb{R})$.
Then for the corresponding solutions $\{ (\W_n, R_n)\}_{n=0}^\infty$ and  $(\W,R)$ defined on time $[0,T]$,
\begin{equation*}
\lim_{n\rightarrow \infty}\|(\W_n, R_n) -(\W, R) \|_{C^0([0,T]; \mathcal{H}^s)} = 0.
\end{equation*}
\item Uniqueness: The solution is unique for $s>\frac{3}{2}$. 
When $\frac{5}{4}<s\leq \frac{3}{2}$, we don't have direct uniqueness of the solution. 
The solution is unique in the sense that it is the unique limit of the regular solution.
\end{enumerate}

The rest of the article is organized as follows.
In Section \ref{s:Norms}, we recall some estimates in paradifferential calculus and paraproducts.
We then use these estimates to compute bounds for the auxiliary functions.
In Section \ref{s:reduction}, we first paralinearize the water wave system and rewrite the system using Wahl\'{e}n variables. 
Then we symmetrize the system and write the water wave system as a nonlinear paradifferential dispersive equation of order $\frac{3}{2}$.
Next, in Section \ref{s:Estimate}, we prove Theorem \ref{t:EnergyEstimate} and Theorem \ref{t:MainTwo}. We first derive an a priori energy estimate for the water wave system \eqref{e:WR}, then we use the technique of global para-composition to reexpress the water wave system as a paradifferential equation of \eqref{ConstantDispersive} type and derive the Strichartz estimate.
Finally, in  Section \ref{s:Cauchy}, we combine all the results in previous sections and sketch the proof of the local well-posedness of the water wave system.

\section{Paradifferential estimates and water waves related bounds} \label{s:Norms}
In this section, we first list  definitions of norms we will use in this article.
Our analysis relies heavily on the  paradifferential calculus.
In the second part of this section, we recall some paradifferential estimates. 
Many of these  estimates are relatively standard.
They can be found in for instance \cite{MR2931520, MR3260858, MR3585049, MR4072685} or the textbooks \cite{MR2768550, MR2418072}.
In the third part, we compute some water waves related bounds for auxiliary functions.

\subsection{Norms and function spaces}
To begin with, we recall the  Littlewood-Paley frequency decomposition,
\begin{equation*}
    I = \sum_{k\in \mathbb{N}} P_k, 
\end{equation*}
where for each $k\geq 1$, $P_k$ is the smooth frequency projection  that selects  frequency $2^k$, and $P_0$ selects the low frequency components $|\xi|\leq 1$.
\begin{enumerate}
\item Let $s\in \mathbb{R}$, and $p,q \in [1, \infty]$, the non-homogeneous Besov space $B^s_{p,q}(\mathbb{R})$ is  defined as the space of all  tempered distributions $u$ satisfying
\begin{equation*}
\| u\|_{B^s_{p,q}} : = \left\|(2^{ks}\|P_k u \|_{L^p})_{k=0}^\infty \right\|_{l^q} < +\infty.
\end{equation*}
\item When $p = q = \infty$, the Besov space $B^s_{\infty, \infty}$ becomes the \textit{Zygmund space} $C^s_{*}$.
When $p = q =2$, the Besov space $B^s_{2,2}$ becomes the \textit{Sobolev space} $H^s$.
\item Let $1\leq p_1 \leq p_2 \leq \infty$, $1\leq r_1 \leq r_2 \leq \infty$, then for any real number $s$,
\begin{equation*}
    B^s_{p_1, r_1}(\mathbb{R}) \hookrightarrow B^{s-(\frac{1}{p_1} - \frac{1}{p_2})}_{p_2, r_2}(\mathbb{R}).
\end{equation*}
As a special case when $p_1 = r_1 =2$ and $p_2 = r_2 = \infty$, 
\begin{equation}
 H^{s+\frac{1}{2}}(\mathbb{R}) \hookrightarrow C^s_{*}(\mathbb{R}) \quad \forall s, \label{HsCsEmbed}
\end{equation}
the Sobolev space $H^{s+\frac{1}{2}}(\mathbb{R})$ can be embedded into the Zygmund space $C^s_{*}(\mathbb{R})$.
\item Let $k\in \mathbb{N}$, $W^{k,\infty}(\mathbb{R})$ denotes the space of all functions such that $\partial_x^j u \in L^\infty(\mathbb{R})$, $0\leq j \leq k$. 
For $\rho = k+ \sigma$ with $k\in \mathbb{N}$ and $\sigma \in (0,1)$,  $W^{\rho, \infty}(\mathbb{R})$  denotes the space of all function $u\in W^{k,\infty}(\mathbb{R})$ such that  $\partial_x^k u$ is $\sigma$- H\"{o}lder continuous on $\mathbb{R}$. 
\item The Zygmund space $C^s_{*}(\mathbb{R})$ coincides with the H\"{o}lder space $W^{s, \infty}(\mathbb{R})$ when $s\in (0,\infty)\backslash \mathbb{N}$.
One has the embedding properties
\begin{align*}
  &C_{*}^s(\mathbb{R}) \hookrightarrow L^\infty(\mathbb{R}), \quad s>0; \qquad L^\infty(\mathbb{R}) \hookrightarrow C_{*}^s, \quad s<0;\\
  &C_{*}^{s_1}(\mathbb{R})\hookrightarrow C_{*}^{s_2}(\mathbb{R}), \quad H^{s_1}(\mathbb{R})\hookrightarrow H^{s_2}(\mathbb{R}), \qquad s_1>s_2.
\end{align*}
\end{enumerate}

\subsection{Paradifferential estimates}\label{ss:para-estimates}
\begin{definition}
\begin{enumerate}
\item Let $\rho\in [0,\infty)$, $m\in \mathbb{R}$. 
$\Gamma^m_\rho(\mathbb{R})$ denotes the space of locally bounded functions $a(x, \xi)$ on $\mathbb{R}\times (\mathbb{R}\backslash \{0\})$, which are $C^\infty$ with respect to $\xi$ for $\xi \neq 0$ and such that for all $k \in \mathbb{N}$ and $\xi \neq 0$, the function $x\mapsto \partial_\xi^k a(x,\xi)$ belongs to $W^{\rho,\infty}(\mathbb{R})$ and there exists a constant $C_k$ with
\begin{equation*}
\forall |\xi|\geq \frac{1}{2}, \quad \|\partial_\xi^k a(\cdot,\xi) \|_{W^{\rho,\infty}} \leq C_k (1+ |\xi|)^{m-k}.
\end{equation*}
Let $a\in \Gamma^m_\rho$,  we define the semi-norm
\begin{equation*}
M^m_{\rho}(a) = \sup_{k \leq \frac{3}{2}+\rho} \sup_{|\xi|\geq \frac{1}{2}} \|(1+ |\xi|)^{k-m}\partial_\xi^k a(\cdot,\xi)  \|_{W^{\rho,\infty}}.
\end{equation*}
\item Given $a\in \Gamma^m_\rho(\mathbb{R})$, let $C^\infty$ functions $\chi(\theta, \eta)$ and $\psi(\eta)$ be such that for some $0<\epsilon_1 < \epsilon_2<1$,
\begin{align*}
    &\chi(\theta, \eta) = 1,  \text{ if } |\theta| \leq \epsilon_1(1+ |\eta|), \qquad \chi(\theta, \eta) = 0,  \text{ if } |\theta| \geq \epsilon_2(1+ |\eta|),\\
    &\psi(\eta) = 0, \text{ if } |\eta|\leq \frac{1}{5}, \qquad \psi(\eta) =1, \text{ if } |\eta|\geq \frac{1}{4}.
\end{align*}
We define the paradifferential operator $T_a$ by
\begin{align*}
    \widehat{T_a u}(\xi) = \frac{1}{2\pi}\int \chi(\xi -\eta, \eta) \hat{a}(\xi-\eta, \eta)\psi(\eta)\hat{u}(\eta) d\eta,
\end{align*}
where $\hat{a}(\theta, \xi)$ is the Fourier transform of a with respect to the  variable x.
\item Let $m\in \mathbb{R}$, an operator  is said to be of order $m$ if, for all $s\in \mathbb{R}$, it is bounded from $H^s$ to $H^{s-m}$.
\item  Let $\rho\in (-\infty, 0)$, $m\in \mathbb{R}$. 
$\Gamma^m_\rho(\mathbb{R})$ denotes the space of distributions $a(x, \xi)$ on $\mathbb{R}\times (\mathbb{R}\backslash \{0\})$, which are $C^\infty$ with respect to $\xi$ for $\xi \neq 0$ and such that for all $k \in \mathbb{N}$ and $\xi \neq 0$, the function $x\mapsto \partial_\xi^k a(x,\xi)$ belongs to $C^{\rho}_* (\mathbb{R})$ and there exists a constant $C_k$ with
\begin{equation*}
\forall |\xi|\geq \frac{1}{2}, \quad \|\partial_\xi^k a(\cdot,\xi) \|_{C^{\rho}_*} \leq C_k (1+ |\xi|)^{m-k}.
\end{equation*}
Let $a\in \Gamma^m_\rho$,  we define the semi-norm
\begin{equation*}
M^m_{\rho}(a) = \sup_{k \leq \frac{3}{2}+|\rho|} \sup_{|\xi|\geq \frac{1}{2}} \|(1+ |\xi|)^{k-m}\partial_\xi^k a(\cdot,\xi)  \|_{C^{\rho}_*}.
\end{equation*}
\end{enumerate}
\end{definition}

We recall the basic symbolic calculus for paradifferential operators in the following result, these can be found in \cite{MR2418072}.
\begin{lemma}[Symbolic calculus, \cite{MR2418072}]
Let $m\in \mathbb{R}$ and $\rho\in [0, +\infty)$.
\begin{enumerate}
\item If $a\in \Gamma^m_0$, then the paradifferential operator $T_a$ is of order m. Moreover, for all $s\in \mathbb{R}$, there exists a positive constant $K$ such that
\begin{equation}
\|T_a\|_{H^s\rightarrow H^{s-m}} \leq K M^m_0(a). \label{TABound}
\end{equation}
\item If $a\in \Gamma^m_\rho$, and $b\in \Gamma^{m^{'}}_\rho$ with $\rho>0$, then the operator $T_a T_b -T_{a\sharp b}$ is of order $m+m^{'}-\rho$, where the composition
\begin{equation*}
    a \sharp b := \sum_{\alpha<\rho} \frac{(-i)^\alpha}{\alpha !} \partial^\alpha_\xi a(x,\xi) \partial^\alpha_x b(x,\xi).
\end{equation*}
Moreover,  for all $s\in \mathbb{R}$, there exists a positive constant $K$ such that
\begin{equation}
  \|T_a T_b-T_{a\sharp b} \|_{H^s \rightarrow H^{s-m-m^{'}+\rho}} \leq K\left(M^m_\rho(a)M^{m^{'}}_0(b) + M^m_0(a)M^{m^{'}}_\rho(b)\right). \label{CompositionPara}
\end{equation}
\item Let $a\in \Gamma^m_\rho$ with $\rho > 0$. 
Denote by $(T_a)^{*}$ the adjoint operator of $T_a$ and by $\bar{a}$ the complex conjugate of $a$.
Then $(T_a)^{*} -T_{a^{*}}$ is of order $m - \rho$, where 
\begin{equation*}
    a^{*} =  \sum_{\alpha<\rho} \frac{1}{i^\alpha \alpha !} \partial^\alpha_\xi \partial^\alpha_x \bar{a}.
\end{equation*}
Moreover,  for all $s\in \mathbb{R}$, there exists a positive constant $K$ such that
\begin{equation}
 \|(T_a)^{*} -T_{a^{*}} \|_{H^s \rightarrow H^{s-m+\rho}} \leq KM^m_\rho(a). \label{AdjointBound}
\end{equation}
\end{enumerate}
\end{lemma}

When $\rho<0$, we recall Proposition $2.12$ of Alazard-Burq-Zuily \cite{MR3260858}.
\begin{lemma} [\hspace{1sp}\cite{MR3260858}] \label{t:NegativeRhoCal}
Let $\rho<0$, $m\in \mathbb{R}$ and $a\in \Gamma^m_\rho$. 
Then the operator $T_a$ is of order $m-\rho$, and satisfies
\begin{equation*}
    \|T_a\|_{H^s\rightarrow H^{s-(m-\rho)}} \leq C M^m_\rho(a), \quad \|T_a\|_{C^s_* \rightarrow C^{s-(m-\rho)}_{*}} \leq   C M^m_\rho(a).
\end{equation*} 
\end{lemma}

In Besov spaces, similar results hold for symbolic calculus.
\begin{lemma}[\hspace{1sp}\cite{MR3585049}]
Let $m, m^{'}, s\in \mathbb{R}$, $q\in [1,\infty]$ and $\rho\in [0,1]$.
\begin{enumerate}
\item If $a\in \Gamma^m_0$, then there exists a positive constant $K$ such that
\begin{equation*}
\|T_a\|_{B^s_{\infty, q}\rightarrow B^{s-m}_{\infty, q}} \leq KM^m_0(a).
\end{equation*}
\item If $a\in \Gamma^m_\rho$, and $b\in \Gamma^m_\rho$, then there exists a positive constant $K$ such that
\begin{equation}
  \|T_a T_b-T_{a b} \|_{B^s_{\infty, q} \rightarrow B^{s-m-m^{'}+\rho}_{\infty, q}} \leq K\left(M^m_\rho(a)M^{m^{'}}_0(b) + M^m_0(a)M^{m^{'}}_\rho(b)\right). \label{CompositionTwo}
\end{equation}
\end{enumerate}
In particular, when $q = \infty$, above symbolic calculus results hold for Zygmund space $C^s_{*}$.
\end{lemma}

When $a$ is only a function of $x$, $T_a u$ is the low-high paraproduct.
We then define 
\begin{equation*}
\Pi(a, u) := au -T_a u -T_u a
\end{equation*}
to be the balanced paraproduct.
For later use, we record below some estimates for paraproducts.

\begin{lemma}[\hspace{1sp}\cite{MR2768550}]
\begin{enumerate}
\item Let $\alpha, \beta \in \mathbb{R}$. 
If $\alpha+ \beta >0$, then
\begin{align}
&\|\Pi(a, u)\|_{H^{\alpha + \beta}(\mathbb{R})} \lesssim \| a\|_{C_{*}^\alpha(\mathbb{R})} \| u\|_{H^\beta(\mathbb{R})}, \label{HCHEstimate}\\
& \|\Pi(a, u)\|_{C_{*}^{\alpha + \beta}(\mathbb{R})} \lesssim \| a\|_{C_{*}^\alpha(\mathbb{R})} \| u\|_{C_{*}^\beta(\mathbb{R})}  \label{CCCEstimate}
\end{align}
\item Let $s_0, s_1, s_2$ be such that $s_0 \leq s_2$ and $s_0< s_1 +s_2 -\frac{1}{2}$, then
\begin{equation}
    \|T_a u\|_{H^{s_0}(\mathbb{R})} \lesssim \| u\|_{H^{s_1}(\mathbb{R})} \|u \|_{H^{s_2}(\mathbb{R})}.
\end{equation}
If in addition to the conditions above, $s_1 +s_2 >0$, then
\begin{equation}
\|au -T_a u \|_{H^{s_0}(\mathbb{R})} \lesssim \| u\|_{H^{s_1}(\mathbb{R})} \|u \|_{H^{s_2}(\mathbb{R})}.
\end{equation}
\item Let $m > 0$ and $s\in \mathbb{R}$, then
\begin{align}
&\|T_{a} u \|_{H^{s-m}(\mathbb{R})} \lesssim \|a\|_{C_{*}^{-m}(\mathbb{R})} \| u \|_{H^s(\mathbb{R})} \label{HsCmStar}, \\
&\|T_{a} u \|_{H^{s}(\mathbb{R})} \lesssim \|a\|_{L^\infty(\mathbb{R})} \| u \|_{H^s(\mathbb{R})}, \label{HsLinfty}\\
&\|T_{a} u \|_{C_{*}^{s-m}(\mathbb{R})} \lesssim \|a\|_{C_{*}^{-m}(\mathbb{R})} \| u \|_{C_{*}^s(\mathbb{R})}, \label{CsCmStar}\\
&\|T_{a} u \|_{C_{*}^{s}(\mathbb{R})} \lesssim \|a\|_{L^\infty(\mathbb{R})} \| u \|_{C_{*}^s(\mathbb{R})}. \label{CsLInfty}
\end{align}
\end{enumerate}
\end{lemma}

Using the above paraproducts estimates, we get the following results.
\begin{lemma}[\hspace{1sp}\cite{MR2768550}]
\begin{enumerate}
\item If $s \geq 0$, then
\begin{align}
&\|uv\|_{H^s(\mathbb{R})} \lesssim \|u\|_{H^s(\mathbb{R})}\|v\|_{L^\infty(\mathbb{R})}+ \|u\|_{L^\infty(\mathbb{R})}\|v\|_{H^s(\mathbb{R})},\label{HsProduct} \\
&\|uv\|_{C_{*}^s(\mathbb{R})} \lesssim \|u\|_{C_{*}^s(\mathbb{R})}\|v\|_{L^\infty(\mathbb{R})}+ \|u\|_{L^\infty(\mathbb{R})}\|v\|_{C_{*}^s(\mathbb{R})}. \label{CsProduct}
\end{align}    
\item  Let $s_0, s_1, s_2$ be such that $s_1 +s_2 >0$, $s_0 \leq s_1$, $s_0 \leq s_2$, and $s_0<s_1+s_2-\frac{1}{2}$, then
\begin{equation}
    \|uv\|_{H^{s_0}(\mathbb{R})} \lesssim \|u\|_{H^{s_1}(\mathbb{R})} \|v\|_{H^{s_2}(\mathbb{R})}.
\end{equation}
\item Let  a smooth function $F\in C^\infty(\mathbb{C}^N)$ satisfying $F(0) = 0$.
There exists a nondecreasing function $\mathcal{F}: \mathbb{R}_{+} \rightarrow \mathbb{R}_{+}$ such that,
\begin{align}
&\|F(u) \|_{H^s} \leq \mathcal{F}(\|u\|_{L^\infty}) \|u\|_{H^s},\quad s\geq 0,  \label{MoserOne}\\
&\|F(u) \|_{C_{*}^s} \leq \mathcal{F}(\|u\|_{L^\infty}) \|u\|_{C_{*}^s},\quad s>0. \label{MoserTwo}
\end{align}
\item Let $s_1 > s_2 > 0$, then
\begin{equation}
    \|uv\|_{C^{-s_2}_*}\lesssim \|u\|_{C^{s_1}_*} \|v\|_{C^{s_2}_*}. \label{CNegativeAlpha}
\end{equation}
\end{enumerate}
\end{lemma}

Later, when we paralinearize the capillary terms in the water wave system, we need the following result of paralinearization.
\begin{lemma}[Paralinearization \cite{MR3770970}] \label{t:Paralinear}
Let $s, \rho>0$, and $F(u)$ be a smooth function of $u$, then for any $u\in H^s(\mathbb{R}^d)\cap C^\rho_{*}(\mathbb{R}^d)$,
\begin{equation*}
    \|F(u)-F(0)-T_{F^{'}(u)}u\|_{H^{s+\rho}(\mathbb{R}^d)} \leq C(\|u\|_{L^\infty(\mathbb{R}^d)})\|u\|_{C^\rho_{*}(\mathbb{R}^d)}\|u\|_{H^s(\mathbb{R}^d)}.
\end{equation*}
\end{lemma}
We remark that this lemma also works for multivariable functions $F$.
We simply need to replace $F^{'}$ by partial derivatives of $F$.
See for instance Lemma $3.26$ in Alazard-Burq-Zuily ~\cite{MR2805065}.

We also record here Lemma $B.7$ of Zhu \cite{MR4072685} on the real (or imaginary) part of the paradifferential operators.
\begin{lemma}[\hspace{1sp}\cite{MR4072685}] \label{t:RealIm}
Let $a(x,\xi)$ be a symbol such that it is either a real-valued even function of $\xi$, or a pure imaginary-valued odd function of $\xi$.
Then for $u\in C^\infty(\mathbb{R}, \mathbb{C})$,
\begin{equation*}
a(x, D)\Re u = \Re a(x,D) u, \quad T_a \Re u = \Re T_a u.
\end{equation*}
Similar results hold when taking the imaginary parts. 
\end{lemma}

To define the para-composition operator later in Section \ref{s:Strichartz}, we also consider a homogeneous Littlewood-Paley frequency composition,
\begin{equation*}
    I = \sum_{k\in \mathbb{Z}} \tilde{P}_k, 
\end{equation*}
where  $\tilde{P}_k$ is the smooth frequency projection  that selects  frequency $2^k$.
The \textit{truncated paradifferential operator} $\dot{T}_a$ introduced in Nguyen \cite{MR3724757} is defined by
\begin{equation*}
  \dot{T}_a u = \sum_{k=1}^{+\infty} \sum_{j=-\infty}^{k-N} \tilde{P}_j a \tilde{P}_k u,  
\end{equation*}
for some fixed positive integer $N$. 
$\dot{T}_a u$ satisfies the same estimates as $T_a u$.

In Section \ref{s:Strichartz}, we will use Lemma $3.2$ from Alazard-Burq-Zuily \cite{MR2931520} to estimate the composition of a function with a diffeomorphism.
\begin{lemma}[\hspace{1sp}\cite{MR2931520}] \label{t:Composition}
Let $p\in \mathbb{N}^{*}$, and $\kappa : \mathbb{R} \rightarrow \mathbb{R}$  be a diffeomorphism such that $\partial_x \kappa \in W^{p-1, \infty}(\mathbb{R})$.
Set $\chi = \kappa^{-1}$,  then for all $F\in H^s(\mathbb{R})$ with $0\leq s \leq p$,  we have $F\circ \kappa \in H^s(\mathbb{R})$, and 
\begin{equation*}
    \|F\circ \kappa \|_{H^s} \leq \|\chi^{'} \|_{L^\infty} C\left( \|\partial_x \kappa \|_{W^{p-1, \infty}} \right) \|F\|_{H^s},
\end{equation*}
where $C(x)$ is an increasing function from $\mathbb{R}^+$ to $\mathbb{R}^+$.
\end{lemma}

\subsection{Water waves related bounds}
In this subsection, we  compute estimates  for the various auxiliary functions including $a, \ub$, $Y$, and estimates for some functions involving the Jacobian $J$ using the paradifferential estimates in the previous subsection.
Some of the estimates are similar to those derived in \cite{ai2023dimensional}.

We first compute the estimates for the  frequency-shift $\ua$.
Recall that it is given by
\begin{equation*}
\ua = i (\bar{\nP}[\bar{R}R_\alpha] - \nP[R\bar{R}_\alpha]) + \frac{\gamma}{2}(R+ \bar{R}- \nP[W\bar{R}_\alpha - \bar{\W}R]- \bar{\nP}[\bar{W}R_\alpha - \W\bar{R}]). 
\end{equation*}
\begin{lemma}
The  frequency-shift $\ua$ satisfies the estimate
\begin{equation}
\|\ua\|_{C_{*}^{\frac{1}{2}}} \lesssim_\CalA \CalB, \label{uACHalf}
\end{equation}
as well as the Sobolev estimate
\begin{equation}
\| \ua\|_{H^s}  \lesssim \CalB \|(\W,R)\|_{\mathcal{H}^{s}}, \quad s>0. \label{UaHsEst}  
\end{equation}
\end{lemma}
\begin{proof}
For the holomorphic terms, we write
\begin{align*}
&\nP[R\bar{R}_\alpha] = T_{\bar{R}_\alpha} R + \nP\Pi(\bar{R}_\alpha, R),\\ 
&\nP[W\bar{R}_\alpha - \bar{\W}R] = T_{\bar{R}_\alpha}W -T_{\bar{\W}}R + \nP\Pi(W, \bar{R}_\alpha)-\nP\Pi(\bar{\W}, R).
\end{align*}
Using \eqref{CCCEstimate} and \eqref{CsLInfty}, 
\begin{align*}
&\|T_{\bar{R}_\alpha} R \|_{C_{*}^{\frac{1}{2}}}+ \|\nP\Pi(\bar{R}_\alpha, R) \|_{C_{*}^{\frac{1}{2}}} \lesssim \|R \|_{C_{*}^{\frac{1}{2}}} \| R\|_{C^1_{*}} \lesssim_\CalA \CalB, \\
 \gamma &\|T_{\bar{R}_\alpha} W \|_{C_{*}^{\frac{1}{2}}}+ \gamma\|\nP\Pi(\bar{R}_\alpha, W) \|_{C_{*}^{\frac{1}{2}}} \lesssim \gamma \|W \|_{C_{*}^{\frac{1}{2}}} \| R\|_{C^1_{*}} + \gamma \|R\|_{C_{*}^{\frac{1}{2}}}\| W\|_{C_{*}^{1}} \lesssim_\CalA \CalB, \\
 \gamma &\|T_{\bar{\W}}R \|_{C_{*}^{\frac{1}{2}}}+ \gamma\|\nP\Pi(\bar{\W}, R) \|_{C_{*}^{\frac{1}{2}}} \lesssim \gamma \|W \|_{C_{*}^{\frac{1}{2}}} \| R\|_{C^1_{*}}\lesssim_\CalA \CalB.
\end{align*}
These give the bound \eqref{uACHalf} for the frequency-shift $\ua$.
For the Sobolev estimate, we use \eqref{HCHEstimate}, \eqref{HsCmStar}, \eqref{HsLinfty} and embeddings to estimate
\begin{align*}
&\|T_{\bar{R}_\alpha} R \|_{H^{s}}+ \|\nP\Pi(\bar{R}_\alpha, R) \|_{H^s} \lesssim \|R \|_{C_{*}^{1}} \| R\|_{H^s} \lesssim \CalB \| R\|_{H^s}, \\
 \gamma &\|T_{\bar{R}_\alpha} W \|_{H^s}+ \gamma\|\nP\Pi(\bar{R}_\alpha, W) \|_{H^s} \lesssim \gamma \|W \|_{H^{s+\frac{1}{2}}} \| R_\alpha\|_{C^{-\frac{1}{2}}_{*}} \lesssim \CalB \| \W\|_{H^{s+\frac{1}{2}}}, \\
 \gamma &\|T_{\bar{\W}}R \|_{H^s}+ \gamma\|\nP\Pi(\bar{\W}, R) \|_{H^s} \lesssim \gamma \|W \|_{C_{*}^{1+\epsilon}} \| R\|_{H^s}\lesssim \CalB \| R\|_{H^s}.    
\end{align*}
These give the Sobolev bound \eqref{UaHsEst} for the frequency-shift $\ua$.
\end{proof}

We continue with the advection velocity $\ub$.
From its definition, $\ub$ can be rewritten using $Y = \frac{\W}{1+\W}$,
\begin{equation*}
\ub = 2\Re R-\nP[R\bar{Y}] - \bar{\nP}[\bar{R}Y] +\gamma \Im W + i\frac{\gamma}{2}(\nP[W\bar{Y}]-\bar{\nP}[\bar{W}Y]).
\end{equation*}
The estimates for $\ub$ again follow in a similar spirit from \cite{ai2023dimensional}.
\begin{lemma}
The advection velocity $\ub$ satisfies the estimate
\begin{equation}
\| \ub\|_{C^1_{*}} \lesssim_\CalA \CalB, \label{UbCOneStar}
\end{equation}
as well as the Sobolev estimate
\begin{equation}
\| \ub\|_{H^s}  \lesssim_{\CalA} \|(\W,R)\|_{\mathcal{H}^{s}}, \quad s>0. \label{UbHsEst}  
\end{equation}
\end{lemma}
\begin{proof}
The bounds for $\Re R$ and $\gamma\Im W$ are obvious, it suffices to estimate $\nP[R\bar{Y}]$ and $i\frac{\gamma}{2}\nP[W\bar{Y}]$ terms.
For these two terms, 
\begin{equation*}
\nP[R\bar{Y}] = T_{\bar{Y}}R + \nP \Pi(R,\bar{Y}), \quad \nP[W\bar{Y}] = T_{\bar{Y}}W + \nP \Pi(W,\bar{Y}). 
\end{equation*}
Using the estimates \eqref{CCCEstimate} and \eqref{CsLInfty},
\begin{align*}
&\|\nP[R\bar{Y}] \|_{C^1_{*}} \leq \|T_{\bar{Y}}R \|_{C^1_{*}} + \|\nP \Pi(R,\bar{Y}) \|_{C^1_{*}} \lesssim \|Y\|_{L^\infty}\|R\|_{C^1_{*}}  \lesssim_\CalA \CalB,\\
\gamma&\|\nP[W\bar{Y}] \|_{C^1_{*}} \leq \gamma\|T_{\bar{Y}}W \|_{C^1_{*}} + \gamma\|\nP \Pi(W,\bar{Y}) \|_{C^1_{*}} \lesssim \gamma \|Y\|_{C^0_{*}}\|W\|_{C^1_{*}}  \lesssim_\CalA \CalB,
\end{align*}
which gives the $C^1_{*}$ estimate \eqref{UbCOneStar}.
For the Sobolev estimate, we use \eqref{HsLinfty} and \eqref{HCHEstimate},
\begin{align*}
&\|\nP[R\bar{Y}] \|_{H^s} \leq \|T_{\bar{Y}}R \|_{H^s} + \|\nP \Pi(R,\bar{Y}) \|_{H^s} \lesssim \|Y\|_{L^\infty\cap C^0_{*}}\|R\|_{H^s}  \lesssim_\CalA \|R\|_{H^s},\\
\gamma&\|\nP[W\bar{Y}] \|_{H^s} \leq \gamma\|T_{\bar{Y}}W \|_{H^s} + \gamma\|\nP \Pi(W,\bar{Y}) \|_{H^s} \lesssim \gamma \|Y\|_{L^\infty\cap C^0_{*}}\|W\|_{H^s}  \lesssim_\CalA  \|W\|_{H^s},
\end{align*}
which gives the Sobolev bound \eqref{UbHsEst}.
\end{proof}

As for the auxiliary function $Y$, one has the following estimates, that are akin to the corresponding estimates in \cite{ai2023dimensional}:
\begin{lemma}
For $s>0$, the auxiliary function $Y = \frac{\W}{1+\W}$ satisfies
\begin{equation}
\|Y\|_{H^s} \lesssim_\CalA \|\W\|_{H^s}, \quad \|Y\|_{C^s_{*}} \lesssim_\CalA \|\W\|_{C^s_{*}}. \label{YMoser}
\end{equation}
Moreover, one can write
\begin{equation}
    Y = T_{(1-Y)^2}\W + E, \label{YWExpression}
\end{equation}
where the error $E$ satisfies the bound
\begin{equation*}
\|E\|_{H^s} \lesssim_\CalA \CalB \|\W\|_{H^{s-1}}.
\end{equation*}
\end{lemma}

\begin{proof}
\eqref{YMoser} is a direct consequence of the Moser type inequalities \eqref{MoserOne} and \eqref{MoserTwo}.
We write $Y$ in  the paraproduct expansion,
\begin{equation*}
    Y = (1-Y)\W = T_{1-Y}\W -T_{\W}Y -\Pi(\W, Y),
\end{equation*}
so that 
\begin{equation*}
T_{1+\W}Y = T_{1-Y}\W -\Pi(\W, Y).
\end{equation*}
Using the fact that $(1+\W)(1-Y) = 1$, one can write
\begin{equation*}
Y - T_{(1-Y)^2}\W = (T_{(1-Y)(1+\W)}-T_{1-Y}T_{1+\W})\W + (T_{1-Y}T_{1-Y}-T_{(1-Y)^2})\W - T_{1-Y}\Pi(\W, Y).
\end{equation*}
If suffices to show that the right-hand side of above equation may be put into the error $E$.
This is due to the composition formula \eqref{CompositionPara} with $\rho =1$ and \eqref{HCHEstimate},
\begin{align*}
 &\|(T_{(1-Y)(1+\W)}-T_{1-Y}T_{1+\W})\W \|_{H^s} \lesssim_\CalA \CalB \|\W\|_{H^{s-1}}, \\
 &\|(T_{1-Y}T_{1-Y}-T_{(1-Y)^2})\W \|_{H^s}\lesssim_\CalA \CalB \|\W\|_{H^{s-1}}, \\
 &\|T_{1-Y}\Pi(\W, Y)\|_{H^s} \leq (1+\|Y\|_{L^\infty})\|\Pi(\W, Y) \|_{H^s} \lesssim_\CalA \|Y\|_{C^1_{*}}\|\W\|_{H^{s-1}}.
\end{align*}
Hence, they can be put into the error term $E$, and we get the expression \eqref{YWExpression}.
\end{proof}

Finally, we compute some paralinearization results involving the Jacobian $J= (1+\W)(1+\bar{\W})$.
\begin{lemma}
Let $s>0$, we have the following paralinearization results:
\begin{align}
&\nP(J^{-\frac{1}{2}}-1) = -\frac{1}{2}T_{(1-Y)J^{-\frac{1}{2}}} \W + E_1,  \label{OneParaJ}\\
&\nP(J^{-\frac{1}{2}}(1-Y)-1) = -\frac{3}{2}T_{(1-Y)^2J^{-\frac{1}{2}}}\W + E_2,\label{TwoParaJ}
\end{align}
where the remainder terms satisfy the bounds
\begin{equation*}
 \|E_1 \|_{H^{s+\frac{3}{2}}(\mathbb{R})}+  \|E_2 \|_{H^{s+\frac{3}{2}}(\mathbb{R})}\lesssim_\CalA \| \W\|_{C^{\frac{3}{2}}_{*}} \|\W\|_{H^s(\mathbb{R})}, \quad  \|E_1\|_{C^{\frac{3}{2}}_{*}}+ \|E_2\|_{C^{\frac{3}{2}}_{*}} \lesssim_\CalA \|\W\|_{C^{\frac{3}{2}}_{*}}. 
\end{equation*}
\end{lemma}

\begin{proof}
We write 
\begin{equation*}
J^{-\frac{1}{2}}-1 = F(\W, \bar{\W}), \quad \text{where } F(a,b) = (1+a)^{-\frac{1}{2}}(1+b)^{-\frac{1}{2}} - 1.
\end{equation*}
Direct computation gives
\begin{equation*}
\frac{\partial F}{\partial a}(a, b) = -\frac{1}{2}(1+a)^{-\frac{3}{2}}(1+b)^{-\frac{1}{2}}, \quad \frac{\partial F}{\partial b}(a, b) = -\frac{1}{2}(1+a)^{-\frac{1}{2}}(1+b)^{-\frac{3}{2}}.
\end{equation*}
Using the paralinearization result Lemma \ref{t:Paralinear}, we get 
\begin{equation*}
J^{-\frac{1}{2}}-1 = -\frac{1}{2}T_{(1-Y)J^{-\frac{1}{2}}} \W - \frac{1}{2} T_{(1-\bar{Y})J^{-\frac{1}{2}}} \bar{\W} + E_3, \quad \|E_3\|_{H^{s+\frac{3}{2}}(\mathbb{R})} \lesssim_\CalA \| \W\|_{C^{\frac{3}{2}}_{*}} \|\W\|_{H^s(\mathbb{R})}.
\end{equation*}
The term $-\frac{1}{2}T_{(1-\bar{Y})J^{-\frac{1}{2}}} \bar{\W}$ is anti-holomorphic, and vanishes after applying the projection $\nP$.
In addition, using Moser type inequality \eqref{MoserTwo},
\begin{equation*}
    E_1 = \nP(J^{-\frac{1}{2}}-1) +\frac{1}{2}T_{(1-Y)J^{-\frac{1}{2}}} \W
\end{equation*}
satisfies
\begin{equation*}
 \|E_1\|_{C^{\frac{3}{2}}_{*}} \lesssim_\CalA \|\W\|_{C^{\frac{3}{2}}_{*}}.  
\end{equation*}
Hence we get the paralinearization \eqref{OneParaJ}.

Similarly, one can write
\begin{equation*}
J^{-\frac{1}{2}}(1-\W)-1 = G(\W,\bar{\W}), \quad  \text{where } G(a,b) = (1+a)^{-\frac{3}{2}}(1+b)^{-\frac{1}{2}} -1.   
\end{equation*}
The two partial derivatives of $G(a,b)$ are given by
\begin{equation*}
\frac{\partial G}{\partial a}(a, b) = -\frac{3}{2}(1+a)^{-\frac{5}{2}}(1+b)^{-\frac{1}{2}}, \quad \frac{\partial G}{\partial b}(a, b) = -\frac{1}{2}(1+a)^{-\frac{3}{2}}(1+b)^{-\frac{3}{2}}.   
\end{equation*}
Using Lemma \ref{t:Paralinear}, one writes
\begin{equation*}
J^{\frac{1}{2}}(1+\W)-1 = -\frac{3}{2}T_{(1-Y)^2J^{-\frac{1}{2}}}\W - \frac{1}{2}T_{ J^{-\frac{3}{2}}}\bar{\W} + E_4, \quad \|E_4\|_{H^{s+\frac{3}{2}}(\mathbb{R})} \lesssim_\CalA \| \W\|_{C^{\frac{3}{2}}_{*}} \|\W\|_{H^s(\mathbb{R})}.
\end{equation*}
The term $- \frac{1}{2}T_{ J^{-\frac{3}{2}}}\bar{\W}$ is anti-holomorphic, and vanishes after applying the projection $\nP$.
Furthermore, using Moser type inequality \eqref{MoserTwo},
\begin{equation*}
    E_2 = \nP(J^{-\frac{1}{2}}(1-Y)-1) +\frac{3}{2}T_{(1-Y)^2J^{-\frac{1}{2}}}\W
\end{equation*}
satisfies
\begin{equation*}
 \|E_2\|_{C^{\frac{3}{2}}_{*}} \lesssim_\CalA \|\W\|_{C^{\frac{3}{2}}_{*}}.
\end{equation*}
This gives the second equation \eqref{TwoParaJ}.
\end{proof}

\section{Paralinearization and reduction of the water wave system}\label{s:reduction}
In this section, we first rewrite the water wave system \eqref{e:WR} as a system of paradifferential equations.
Next, we use a change of unknown called \textit{Wahl\'{e}n variables} to rewrite the paradifferential system.
Finally, we symmetrize the water wave system.

\subsection{Paralinearization of the water wave system}
In this subsection, we use the paradifferential estimates and bounds for auxiliary functions in Section \ref{s:Norms} to paralinearize the water wave system \eqref{e:WR}.

Using auxiliary functions $\ua$ and $\ub$, the first equation of \eqref{e:CVWW} can be rewritten as
\begin{equation}
 W_t + \ub(1+ W_\alpha) + i\frac{\gamma}{2}W = \bar{R} + i\frac{\gamma}{2}\bar{W}. \label{WEqn} 
\end{equation}
To begin with, we obtain a paradifferential equation for $W$ using \eqref{WEqn} that is similar to the equation for $W$ in \cite{ai2023dimensional}.
We remark that although the unknown $W$ does not appear directly in the differentiated system \eqref{e:WR}, we still need to paralinearize the equation \eqref{WEqn} because it will play a role in Wahl\'{e}n variables later. 
\begin{lemma} 
Let $s>0$, then we have the following results.
\begin{enumerate}
\item The unknown $W$ solves the paradifferential equation
\begin{equation}
  W_t + T_\ub \partial_\alpha W = -T_{1+W_\alpha} \nP[(1-\bar{Y})R]+ i\frac{\gamma}{2}T_{1+W_\alpha} \nP[W(1-\bar{Y})]-\nP\Pi(W_\alpha, \ub)-i\frac{\gamma}{2}W. \label{ParaW}
\end{equation}
\item The unknown $W$ satisfies the paradifferential equation
\begin{equation}
       W_t + T_\ub \partial_\alpha W +T_{1+\W}T_{1-\bar{Y}}R = G,  \label{WParaMaterial}
\end{equation}
where the source term $G$  satisfies the Sobolev bound
\begin{equation*}
    \|G\|_{H^s} \lesssim_\CalA \CalB \|(\W,R)\|_{\mathcal{H}^{s-\frac{3}{2}}}.
\end{equation*}
\item Furthermore, the equation \eqref{WParaMaterial} can be  reduced to
\begin{equation}
       W_t + T_\ub \partial_\alpha W +R =  \tilde{G},  \label{WParaTwo}
\end{equation}
where the source term $\tilde{G}$ satisfies the Sobolev bound
\begin{equation*}
 \|\tilde{G}\|_{H^{s}} \lesssim_\CalA \CalB \|(\W,R)\|_{\mathcal{H}^{s}}.    
\end{equation*}
\end{enumerate}
\end{lemma}

\begin{proof}
we use  \eqref{WEqn} to write
 \begin{equation*}
     W_t + T_\ub \partial_\alpha W = -T_{1+W_\alpha} \ub -\Pi(W_\alpha,\ub)- i\frac{\gamma}{2}W + \bar{R}+i\frac{\gamma}{2}\bar{W}.
 \end{equation*}
 After plugging in the expression of $\ub$ and applying the holomorphic projection $\nP$, we can eliminate the anti-holomorphic portion and obtain the paradifferential equation \eqref{ParaW}.

 For the paradifferential equation of $W$, we write
\begin{equation}
 \begin{aligned}
  W_t + T_\ub \partial_\alpha W +T_{1+\W}T_{1-\bar{Y}}R =& T_{1+\W}\nP\Pi(\bar{Y}, R) - \nP\Pi(\W, \ub) \\
  +&i\frac{\gamma}{2}T_{\W}W - i\frac{\gamma}{2}T_{1+\W}T_{\bar{Y}}W - i\frac{\gamma}{2}T_{1+\W}\nP\Pi(W,\bar{Y}). 
 \end{aligned}     \label{WFormula}
\end{equation}
We use \eqref{HCHEstimate}, \eqref{HsLinfty}, the Moser type estimate \eqref{MoserTwo}, and the bound \eqref{UbCOneStar} to estimate  the right-hand side of \eqref{WFormula},
\begin{align*}
&\|T_{1+\W}\nP\Pi(\bar{Y}, R) \|_{H^s} \lesssim_\CalA \|\Pi(\bar{Y},R)\|_{H^s} \lesssim_\CalA \|Y \|_{C^{\frac{3}{2}}_{*}} \|R \|_{H^{s-\frac{3}{2}}}\lesssim_\CalA \|\W \|_{C^{\frac{3}{2}}_{*}} \|R \|_{H^{s-\frac{3}{2}}}, \\
&\|\nP\Pi(\W, \ub) \|_{H^s} \lesssim \| \W\|_{H^{s-1}}\|\ub\|_{C^1_{*}} \lesssim_\CalA \CalB \| \W\|_{H^{s-1}},\\
& \gamma \|T_\W W\|_{H^s} \lesssim \gamma \|\W\|_{L^\infty} \|\W \|_{H^{s-1}}\lesssim_\CalA \CalB \| \W\|_{H^{s-1}},\\
& \gamma  \| T_{1+\W}T_{\bar{Y}}W\|_{H^s}\lesssim_\CalA \gamma\|T_{\bar{Y}}W \|_{H^s}\lesssim_\CalA \gamma \|Y\|_{L^\infty} \|\W \|_{H^{s-1}}\lesssim_\CalA \CalB \| \W\|_{H^{s-1}},\\
& \gamma \|T_{1+\W}\nP\Pi(W,\bar{Y}) \|_{H^s}\lesssim_\CalA \gamma \|\Pi(W,\bar{Y}) \|_{H^s}\lesssim_\CalA  \gamma \|Y\|_{L^\infty} \|\W \|_{H^{s-1}}\lesssim_\CalA \CalB \| \W\|_{H^{s-1}}.
\end{align*}
As a consequence, the right-hand side of \eqref{WFormula} can be absorbed into the source term $G$.

We further rewrite 
\begin{equation*}
 W_t + T_\ub \partial_\alpha W +R =  G + T_{\bar{Y}}R - T_{\W}R -T_{\W}T_{\bar{Y}}R.    
\end{equation*}
By \eqref{HsLinfty}, we see that the right-hand side of the above equation satisfies the desired bound and can be put into $\tilde{G}$.
\end{proof}

Using the equation \eqref{WParaMaterial}, we now obtain a paradifferential equation for $\W$.
\begin{lemma}
For $s>0$, the unknown $\W$ satisfies the paradifferential equation
\begin{equation}
       \W_t + T_\ub \partial_\alpha \W +\partial_\alpha T_{(1+\W)(1-\bar{Y})}R = G_1,  \label{WParaMat}
\end{equation}
where the source term $G_1$  satisfies the Sobolev bound
\begin{equation*}
    \|G_1\|_{H^s} \lesssim_\CalA \CalB \|(\W,R)\|_{\mathcal{H}^{s-\frac{1}{2}}}.
\end{equation*}
\end{lemma}
\begin{proof}
Taking the $\alpha$ derivative of \eqref{WParaMaterial}, one gets
\begin{equation}
\W_t + T_\ub \partial_\alpha \W +\partial_\alpha T_{(1+\W)(1-\bar{Y})} R = \partial_\alpha (T_{(1+\W)(1-\bar{Y})}-T_{1+\W}T_{1-\bar{Y}})R -T_{\ub_\alpha}\W   +\partial_\alpha G. \label{WAlphaFormula}  
\end{equation}
To estimate the right-hand side of \eqref{WAlphaFormula}, we use \eqref{CompositionPara}, \eqref{HsCmStar}, and \eqref{UbCOneStar}
\begin{align*}
&\|\partial_\alpha (T_{1+\W}T_{1-\bar{Y}}-T_{(1+\W)(1-\bar{Y})})R \|_{H^s} \lesssim_\CalA \CalB \|R\|_{H^{s-\frac{1}{2}}},\\
&\|T_{\ub_\alpha}\W \|_{H^s}\lesssim \|\ub\|_{C^{1}_{*}}\| \W\|_{H^s} \lesssim_\CalA \CalB \| \W\|_{H^s},\\
&\|\partial_\alpha G \|_{H^s}\leq \|G \|_{H^{s+1}} \lesssim_\CalA \CalB \| (\W,R)\|_{\mathcal{H}^{s-\frac{1}{2}}}.
\end{align*}
Therefore, the right-hand side of \eqref{WAlphaFormula} can be absorbed into the source term $G_1$.
\end{proof}

Next, we paralinearize the second equation of the water wave system \eqref{e:WR}.
Note that it can be reexpressed as
\begin{equation}
\begin{aligned}
&R_t + T_{\ub}\partial_\alpha R+i\gamma R = -\nP T_{R_\alpha}\ub -\nP\Pi(R_\alpha, \ub)+i\nP[(g+\ua)Y] -\nP[R\bar{R}_\alpha] \\
-&i\frac{\gamma}{2}\nP[W\bar{R}_\alpha-\bar{W}_\alpha R]- i\sigma (1-Y)\nP\left[ \frac{\W_{ \alpha}}{J^{\frac{1}{2}}(1+\W)}\right]_{\alpha} +i\sigma (1-Y) \nP \left[\frac{\bar{\W}_{ \alpha}}{J^{\frac{1}{2}}(1+\bar{\W})}\right]_{\alpha}. 
\end{aligned}    \label{ParaRExpression}
\end{equation}
For the non-capillary terms on the right-hand side of \eqref{ParaRExpression}, we have the following paralinearization result.
\begin{lemma} 
Let $s>0$, then one can write
\begin{equation}
  -\nP T_{R_\alpha}\ub -\nP\Pi(R_\alpha, \ub)+i\nP[(g+\ua)Y] -\nP[R\bar{R}_\alpha] 
-i\frac{\gamma}{2}\nP[W\bar{R}_\alpha-\bar{\W} R] 
=  ig\W + K,\label{RNonCapillary}
\end{equation}
where the remainder term $K$ satisfies
\begin{equation*}
\| K\|_{H^{s}}\lesssim_\CalA \CalB \|(\W,R)\|_{\mathcal{H}^{s}}.
\end{equation*}
\end{lemma}

\begin{proof}
We first show that
 \begin{equation*}
  i\nP[(g+\ua)Y] = ig\W + K.   
 \end{equation*}
Indeed, by \eqref{YWExpression}, 
\begin{equation*}
i\nP[gY] = ig Y = igT_{(1-Y)^2}\W +K = ig\W +K,
\end{equation*}
and using the product estimate \eqref{HsProduct} and the estimates for $\ua$ \eqref{uACHalf}, \eqref{UaHsEst},
\begin{equation*}
\|\nP[\ua Y] \|_{H^{s}}\lesssim \| \ua\|_{C_{*}^{\frac{1}{2}}} \|Y\|_{H^s} + \| \ua\|_{H^{s}}\|Y\|_{L^\infty} \lesssim_\CalA \CalB \|(\W, R)\|_{\mathcal{H}^{s}}.
\end{equation*}
Then we show that the rest of the terms on the left of \eqref{RNonCapillary} are perturbative and can be put into $K$.
Using \eqref{HCHEstimate}, \eqref{HsLinfty} and the estimate for $\ub$ \eqref{UbHsEst},
\begin{align*}
&\|\nP T_{R_\alpha} \ub \|_{H^{s}}+ \|\nP \Pi(R_\alpha , \ub) \|_{H^{s}} \lesssim \|R_\alpha\|_{C^{\epsilon}_{*}}\|\ub\|_{H^{s}} \lesssim_\CalA \CalB \|(\W, R)\|_{\mathcal{H}^{s}}, \\
& \|\nP[R\bar{R}_\alpha] \|_{H^{s}} \leq \|T_{\bar{R}_\alpha} R\|_{H^{s}} + \|\nP\Pi(\bar{R}_\alpha, R) \|_{H^{s}}\lesssim \| R\|_{C^{1+\epsilon}_{*}}\|R\|_{H^{s}},\\
 \gamma & \|\nP[W\bar{R}_\alpha] \|_{H^{s}} \leq \gamma\|T_{\bar{R}_\alpha}W \|_{H^{s}} + \gamma \|\nP\Pi[W, \bar{R}_\alpha] \|_{H^{s}} \lesssim  \gamma \|R\|_{C^{\frac{1}{2}}_{*}}\|\W\|_{H^s}, \\
 \gamma & \|\nP[R\bar{\W}] \|_{H^{s}}\leq \gamma \|T_{\bar{\W}}R \|_{H^{s}} + \gamma\|\nP\Pi(\bar{\W},R) \|_{H^{s}} \lesssim
 \gamma \|\W\|_{C^\epsilon_{*}}\|R\|_{H^{s}}.
\end{align*}
These terms are all controlled by $\CalB \|(\W,R) \|_{\mathcal{H}^s}$ and can be absorbed into $K$. 
\end{proof}

For two capillary terms on the right-hand side of \eqref{ParaRExpression}, we  apply \eqref{OneParaJ} and \eqref{TwoParaJ} to establish the following intermediate result.
\begin{lemma}
Let $s> 0$, then
\begin{align}
&\nP\left[ \frac{\W_{ \alpha}}{J^{\frac{1}{2}}(1+\W)}\right] = T_{J^{-\frac{1}{2}}(1-Y)} \W_\alpha -\frac{3}{2}T_{(1-Y)^2 J^{-\frac{1}{2}} \W_\alpha } \W   + K_1, \label{CapiOne}\\
&\nP \left[\frac{\bar{\W}_{ \alpha}}{J^{\frac{1}{2}}(1+\bar{\W})}\right]= -\frac{1}{2}T_{\bar{\W}_\alpha J^{-\frac{3}{2}}}\W  +K_2, \label{CapiTwo}
\end{align}
where the remainder terms $K_1$ and $K_2$ satisfy
\begin{equation*}
\| K_1\|_{H^{s+\frac{1}{2}}}+ \| K_2\|_{H^{s+\frac{1}{2}}}\lesssim_\CalA \CalB \|\W\|_{H^s}, \quad \|K_1\|_{C^{\frac{1}{2}}_{*}}+ \|K_2\|_{C^{\frac{1}{2}}_{*}} \lesssim_\CalA \|\W\|_{C^{\frac{3}{2}}_{*}}.
\end{equation*}
\end{lemma}

\begin{proof}
For the first identity, using the paraproduct decomposition and the holomorphic projection $\nP$, we decompose
\begin{equation}
\nP\left[ \frac{\W_{ \alpha}}{J^{\frac{1}{2}}(1+\W)}\right] = T_{J^{-\frac{1}{2}}(1-Y)} \W_\alpha +  T_{\W_\alpha} \nP[J^{-\frac{1}{2}}(1-Y)-1] + \nP\Pi(\W_\alpha, J^{-\frac{1}{2}}(1-Y)-1). \label{FirstCapillary}
\end{equation}
By \eqref{TwoParaJ}, the second term on the right of \eqref{FirstCapillary} $T_{\W_\alpha} \nP[J^{-\frac{1}{2}}(1-Y)-1]$ equals
\begin{equation*}
    -\frac{3}{2}T_{(1-Y)^2 J^{-\frac{1}{2}} \W_\alpha } \W  + T_{\W_\alpha}E_2- \frac{3}{2}\left(T_{\W_\alpha}T_{(1-Y)^2J^{-\frac{1}{2}}} - T_{(1-Y)^2 J^{-\frac{1}{2}} \W_\alpha }\right)\W.
\end{equation*}
Since
\begin{equation*}
 \|T_{\W_\alpha} E_2\|_{H^{s+\frac{1}{2}}}   \lesssim \| \W_\alpha\|_{C^{-1}_{*}}\|E_2 \|_{H^{s+\frac{3}{2}}} \lesssim_\CalA \|\W\|_{C^{\frac{3}{2}}_{*}}\|\W\|_{H^s},  
\end{equation*}
the remainder term
\begin{equation*}
K_3 := T_{\W_\alpha} E_2 - \frac{3}{2}\left(T_{\W_\alpha}T_{(1-Y)^2J^{-\frac{1}{2}}} - T_{(1-Y)^2 J^{-\frac{1}{2}} \W_\alpha }\right)\W
\end{equation*}
satisfies
\begin{equation*}
\|K_3\|_{H^{s+\frac{1}{2}}} \lesssim_\CalA \CalB\|\W \|_{H^s}, \quad \|K_3\|_{C^{\frac{1}{2}}_{*}} \lesssim_\CalA \|\W\|_{C^{\frac{3}{2}}_{*}} 
\end{equation*}
due to the composition estimates \eqref{CompositionPara}, \eqref{CompositionTwo}.
For the balanced term on the right of \eqref{FirstCapillary}, using the Moser type inequalities \eqref{MoserOne}, \eqref{MoserTwo}, it satisfies
\begin{align*}
&\|\nP\Pi(\W_\alpha, J^{-\frac{1}{2}}(1-Y)-1) \|_{H^{s+\frac{1}{2}}} \lesssim \|\W_\alpha\|_{C^{\frac{1}{2}}_{*}}\|J^{-\frac{1}{2}}(1+\W)^{-1}-1 \|_{H^s} \lesssim_\CalA \CalB \|\W\|_{H^s}, \\
& \|\nP\Pi(\W_\alpha, J^{-\frac{1}{2}}(1-Y)-1) \|_{C^{\frac{1}{2}}_{*}} \lesssim \|\W_\alpha\|_{C^{-1}_{*}}\|J^{-\frac{1}{2}}(1+\W)^{-1}-1 \|_{C^{\frac{3}{2}}_{*}} \lesssim_\CalA  \|\W\|_{C^{\frac{3}{2}}_{*}}.
\end{align*}
Hence, by choosing $K_1 = K_3+  \nP\Pi(\W_\alpha, J^{-\frac{1}{2}}(1-Y)-1)$, we get \eqref{CapiOne}. 

Similarly, using the paraproduct decomposition and the holomorphic projection, one writes
\begin{equation}
 \nP \left[\frac{\bar{\W}_{ \alpha}}{J^{\frac{1}{2}}(1+\bar{\W})}\right] =   T_{\frac{\bar{\W}_\alpha}{1+\bar{\W}}} \nP(J^{-\frac{1}{2}}-1) + \nP \Pi\left(\bar{\W}_\alpha(1-\bar{Y}), J^{-\frac{1}{2}}-1\right). \label{SecondCapillary}
\end{equation}
Using \eqref{OneParaJ}, the first term on the right of \eqref{SecondCapillary} equals
\begin{equation*}
T_{\frac{\bar{\W}_\alpha}{1+\bar{\W}}} \nP(J^{-\frac{1}{2}}-1) = -\frac{1}{2}T_{\bar{\W}_\alpha J^{-\frac{3}{2}}}\W +K_4,
\end{equation*}
where the remainder term
\begin{equation*}
K_4 = T_{\frac{\bar{\W}_\alpha}{1+\bar{\W}}} E_1 - \frac{1}{2}\left(T_{\bar{\W}_\alpha J^{-\frac{3}{2}}}- T_{\frac{\bar{\W}_\alpha}{1+\bar{\W}}}T_{(1-Y)J^{-\frac{1}{2}}} \right)\W
\end{equation*}
satisfies
\begin{equation*}
\|K_4\|_{H^{s+\frac{1}{2}}} \lesssim_\CalA \CalB \|\W\|_{H^s}, \quad \|K_4\|_{C^{\frac{1}{2}}_{*}} \lesssim_\CalA \|\W\|_{C^{\frac{3}{2}}_{*}} \|\W\|_{H^s}
\end{equation*}
due to the composition estimates \eqref{CompositionPara}, \eqref{CompositionTwo}. 
For the balanced term on the right of \eqref{SecondCapillary}, using the Moser type inequalities \eqref{MoserOne}, \eqref{MoserTwo}, it satisfies the bounds
\begin{align*}
&\left\|\nP \Pi\left(\bar{\W}_\alpha(1-\bar{Y}), J^{-\frac{1}{2}}-1\right) \right\|_{H^{s+\frac{1}{2}}} \lesssim_\CalA \|\W_\alpha\|_{C^{\frac{1}{2}}_{*}}\|J^{-\frac{1}{2}}-1 \|_{H^s} \lesssim_\CalA \CalB \|\W\|_{H^s},\\
& \left\|\nP \Pi\left(\bar{\W}_\alpha(1-\bar{Y}), J^{-\frac{1}{2}}-1\right) \right\|_{C^{\frac{1}{2}}_{*}} \lesssim \|\W_\alpha\|_{C^{-1}_{*}}\|J^{-\frac{1}{2}}-1 \|_{C^{\frac{3}{2}}_{*}} \lesssim_\CalA  \|\W\|_{C^{\frac{3}{2}}_{*}}.
\end{align*}
Hence, by choosing $K_2 = K_4+   \nP \Pi\left(\bar{\W}_\alpha(1-\bar{Y}), J^{-\frac{1}{2}}-1\right)$, we get \eqref{CapiTwo}.
\end{proof}

Using above paralinearization  results, we then obtain the paralinearization of two capillary terms on the right-hand side of \eqref{ParaRExpression}.
\begin{lemma}
 Let $s>0$, then
 \begin{equation}
 \begin{aligned}
   &- i\sigma (1-Y)\nP\left[ \frac{\W_{ \alpha}}{J^{\frac{1}{2}}(1+\W)}\right]_{\alpha} +i\sigma (1-Y) \nP \left[\frac{\bar{\W}_{ \alpha}}{J^{\frac{1}{2}}(1+\bar{\W})}\right]_{\alpha} \\
   = &-i\sigma T_{J^{-\frac{1}{2}}(1-Y)^2}\W_{\alpha \alpha}+ 3i\sigma T_{J^{-\frac{1}{2}}(1-Y)^3 \W_\alpha} \W_\alpha + K,  
 \end{aligned}
 \label{CurvaturePara}
 \end{equation}
 where the remainder term $K$ satisfies
\begin{equation*}
\| K\|_{H^{s}}\lesssim_\CalA \CalB \|\W\|_{H^{s+\frac{1}{2}}}.
\end{equation*}
\end{lemma}

\begin{proof}
According to the paralinearization results \eqref{CapiOne} and \eqref{CapiTwo}, one can write
\begin{equation*}
 \nP \left[\frac{\bar{\W}_{ \alpha}}{J^{\frac{1}{2}}(1+\bar{\W})} -\frac{\W_{ \alpha}}{J^{\frac{1}{2}}(1+\W)} \right]=  \frac{3}{2}T_{(1-Y)^2 J^{-\frac{1}{2}} \W_\alpha } \W -T_{J^{-\frac{1}{2}}(1-Y)} \W_\alpha -\frac{1}{2}T_{\bar{\W}_\alpha J^{-\frac{3}{2}}}\W  +K_2 -K_1.  
\end{equation*}
Taking the $\alpha$ derivative of both sides, 
\begin{equation}
\begin{aligned}
  &\nP  \left[\frac{\bar{\W}_{ \alpha}}{J^{\frac{1}{2}}(1+\bar{\W})} -\frac{\W_{ \alpha}}{J^{\frac{1}{2}}(1+\W)} \right]_\alpha = -T_{\frac{1-Y}{J^{\frac{1}{2}}}} \W_{\alpha\alpha} -T_{\left(J^{-\frac{1}{2}}(1-Y)\right)_\alpha} \W_{\alpha}  - \frac{1}{2}T_{\bar{\W}_\alpha J^{-\frac{3}{2}}}\W_\alpha\\
  &+\frac{3}{2}T_{(1-Y)^2 J^{-\frac{1}{2}} \W_\alpha } \W_\alpha + \frac{3}{2}T_{\left((1-Y)^2 J^{-\frac{1}{2}} \W_\alpha \right)_\alpha}\W -\frac{1}{2}T_{\left(\bar{\W}_\alpha J^{-\frac{3}{2}}\right)_\alpha}\W + \partial_\alpha (K_2 -  K_1).
\end{aligned}    
\label{ThirdCapillary}
\end{equation}
The second to the fourth terms on the right-hand side of \eqref{ThirdCapillary} can be simplified to
\begin{align*}
  &-T_{\left(J^{-\frac{1}{2}}(1-Y)\right)_\alpha} \W_{\alpha} - \frac{1}{2}T_{\bar{\W}_\alpha J^{-\frac{3}{2}}}\W_\alpha + \frac{3}{2}T_{(1-Y)^2 J^{-\frac{1}{2}} \W_\alpha } \W_\alpha \\
  =& \frac{3}{2}T_{(1-Y)^2 J^{-\frac{1}{2}} \W_\alpha } \W_\alpha + \frac{1}{2}T_{\bar{\W}_\alpha J^{-\frac{3}{2}}}\W_\alpha - \frac{1}{2}T_{\bar{\W}_\alpha J^{-\frac{3}{2}}}\W_\alpha + \frac{3}{2}T_{(1-Y)^2 J^{-\frac{1}{2}} \W_\alpha } \W_\alpha \\
  =& 3T_{(1-Y)^2 J^{-\frac{1}{2}} \W_\alpha } \W_\alpha.
\end{align*}
For the last three terms on the right-hand side of \eqref{ThirdCapillary}, using \eqref{HsCmStar},
\begin{align*}
 &\left \|T_{\left((1-Y)^2 J^{-\frac{1}{2}} \W_\alpha \right)_\alpha}\W \right\|_{H^{s}} \lesssim \|(1-Y)^2 J^{-\frac{1}{2}} \W_\alpha \|_{C^{\frac{1}{2}}_{*}} \| \W\|_{H^{s+\frac{1}{2}}}\lesssim_\CalA \| \W\|_{C^{\frac{3}{2}}_{*}} \| \W\|_{H^{s+\frac{1}{2}}}, \\
 & \left\| T_{\left(\bar{\W}_\alpha J^{-\frac{3}{2}}\right)_\alpha}\W \right\|_{H^s} \lesssim \|\bar{\W}_\alpha J^{-\frac{3}{2}} \|_{C^{\frac{1}{2}}_{*}} \| \W\|_{H^{s+\frac{1}{2}}}\lesssim_\CalA \| \W\|_{C^{\frac{3}{2}}_{*}} \| \W\|_{H^{s+\frac{1}{2}}}, \\
 &\|\partial_\alpha(K_2 - K_1) \|_{H^{s}} \leq \| K_1\|_{H^{s+1}} + \| K_2\|_{H^{s+1}} \lesssim_\CalA \CalB \|\W\|_{H^{s+\frac{1}{2}}}, 
\end{align*}
so that they are perturbative.
Moreover, using \eqref{CsCmStar},
\begin{equation*}
\left \|T_{\left((1-Y)^2 J^{-\frac{1}{2}} \W_\alpha \right)_\alpha}\W \right\|_{C_{*}^{-\frac{1}{2}}} + \left\| T_{\left(\bar{\W}_\alpha J^{-\frac{3}{2}}\right)_\alpha}\W \right\|_{C_{*}^{-\frac{1}{2}}} +\|\partial_\alpha(K_2 - K_1) \|_{C_{*}^{-\frac{1}{2}}} \lesssim_\CalA \|\W\|_{C^{\frac{3}{2}}_{*}}.
\end{equation*}
Hence, 
\begin{equation*}
 \nP  \left[\frac{\bar{\W}_{ \alpha}}{J^{\frac{1}{2}}(1+\bar{\W})}\right]_\alpha  -\nP\left[\frac{\W_{ \alpha}}{J^{\frac{1}{2}}(1+\W)} \right]_\alpha =  -T_{J^{-\frac{1}{2}}(1-Y)}\W_{\alpha \alpha}+ 3 T_{J^{-\frac{1}{2}}(1-Y)^2 \W_\alpha} \W_\alpha + K_5, 
\end{equation*}
where $K_5$ satisfies the estimate
\begin{equation*}
\| K_5\|_{H^{s}}\lesssim_\CalA \CalB \|\W\|_{H^{s+\frac{1}{2}}} , \quad \|K_5\|_{C^{-\frac{1}{2}}_{*}} \lesssim_\CalA \|\W\|_{C^{\frac{3}{2}}_{*}}.
\end{equation*}
Finally, we multiply the above equation with $(1-Y)$ and use the paraproduct decomposition,
\begin{equation}
\begin{aligned}
 &-(1-Y)T_{J^{-\frac{1}{2}}(1-Y)}\W_{\alpha \alpha} = -T_{J^{-\frac{1}{2}}(1-Y)^2}\W_{\alpha \alpha} + T_{T_{J^{-\frac{1}{2}}(1-Y)}\W_{\alpha \alpha}} Y \\
 &+ \Pi\left(T_{J^{-\frac{1}{2}}(1-Y)}\W_{\alpha \alpha}, Y\right)
 + \left(T_{J^{-\frac{1}{2}}(1-Y)^2} - T_{1-Y}T_{J^{-\frac{1}{2}}(1-Y)}\right)\W_{\alpha \alpha}.
\end{aligned}    \label{CapillaryFour}
\end{equation}
For the last three terms of \eqref{CapillaryFour}, applying \eqref{HsCmStar}, \eqref{CsLInfty} and the composition \eqref{CompositionPara},
\begin{align*}
&\left\|T_{T_{J^{-\frac{1}{2}}(1-Y)}\W_{\alpha \alpha}} Y  \right\|_{H^{s}} \lesssim \left\|T_{J^{-\frac{1}{2}}(1-Y)}\W_{\alpha \alpha} \right\|_{C^{-\frac{1}{2}}_{*}} \|Y\|_{H^{s+\frac{1}{2}}} \lesssim_\CalA \|\W\|_{C^{\frac{3}{2}}_{*}} \|\W\|_{H^{s+\frac{1}{2}}},\\
&\left\|\Pi\left(T_{J^{-\frac{1}{2}}(1-Y)}\W_{\alpha \alpha}, Y\right) \right\|_{H^{s}} \lesssim \|T_{J^{-\frac{1}{2}}(1-Y)}\W_{\alpha \alpha} \|_{H^{s-\frac{3}{2}}}\| Y\|_{C^{\frac{3}{2}}_{*}}\lesssim_\CalA \|\W\|_{C^{\frac{3}{2}}_{*}} \|\W\|_{H^{s+\frac{1}{2}}},\\
& \left\| \left(T_{J^{-\frac{1}{2}}(1-Y)^2} - T_{1-Y}T_{J^{-\frac{1}{2}}(1-Y)}\right)\W_{\alpha \alpha} \right\|_{H^{s}} \lesssim \| Y\|_{L^\infty}\| J^{-\frac{1}{2}}(1-Y)-1\|_{W^{\frac{3}{2}, \infty}}\| \W\|_{H^{s+\frac{1}{2}}} \\
&\lesssim_\CalA \|\W\|_{C^{\frac{3}{2}}_{*}} \|\W\|_{H^{s+\frac{1}{2}}}, 
\end{align*}
where we use the Moser type inequality \eqref{MoserTwo}, and the fact that $W^{\frac{3}{2},\infty}(\R) = C^{\frac{3}{2}}_{*}(\R)$.
Hence, these three terms can be put into the remainder $K$.

Similarly, performing the paraproduct decomposition again, one writes
\begin{equation}
\begin{aligned}
&(1-Y)T_{J^{-\frac{1}{2}}(1-Y)^2 \W_\alpha} \W_\alpha =  T_{J^{-\frac{1}{2}}(1-Y)^3 \W_\alpha} \W_\alpha - T_{T_{J^{-\frac{1}{2}}(1-Y)^2 \W_\alpha} \W_\alpha}Y\\
 -& \Pi\left(T_{J^{-\frac{1}{2}}(1-Y)^2 \W_\alpha} \W_\alpha, Y\right)- \left(T_{J^{-\frac{1}{2}}(1-Y)^3 \W_\alpha}- T_{1-Y}T_{J^{-\frac{1}{2}}(1-Y)^2 \W_\alpha}\right)\W_\alpha. 
\end{aligned} \label{CapillaryFive}
\end{equation}
applying \eqref{HsCmStar}, \eqref{CsLInfty}, \eqref{MoserOne} and the composition \eqref{CompositionPara} for the last three terms on the right of \eqref{CapillaryFive},
\begin{align*}
&\left\|T_{T_{J^{-\frac{1}{2}}(1-Y)^2 \W_\alpha} \W_\alpha}Y \right\|_{H^{s}} \lesssim \left\|T_{J^{-\frac{1}{2}}(1-Y)^2 \W_\alpha} \W_\alpha \right\|_{C^{-\frac{1}{2}}_{*}} \| Y\|_{H^{s+\frac{1}{2}}}\lesssim_\CalA \|\W\|_{C^{\frac{3}{2}}_{*}} \|\W\|_{H^{s+\frac{1}{2}}}, \\ 
& \left\| \Pi\left(T_{J^{-\frac{1}{2}}(1-Y)^2 \W_\alpha} \W_\alpha, Y\right) \right\|_{H^{s}} \lesssim \left\|T_{J^{-\frac{1}{2}}(1-Y)^2 \W_\alpha} \W_\alpha  \right\|_{C^{-\frac{1}{2}}_{*}} \| Y\|_{H^{s+\frac{1}{2}}}\lesssim_\CalA \CalB \|\W\|_{H^{s+\frac{1}{2}}}, \\
& \left\| \left(T_{J^{-\frac{1}{2}}(1-Y)^3 \W_\alpha}- T_{1-Y}T_{J^{-\frac{1}{2}}(1-Y)^2 \W_\alpha}\right)\W_\alpha\right\|_{H^{s}} \lesssim_\CalA \CalB \|\W\|_{H^{s+\frac{1}{2}}},
\end{align*}
these three terms can be absorbed into the remainder $K$.

Finally, the term $i\sigma (1-Y)K_5$ is perturbative, since
\begin{align*}
&\|(1-Y)K_5\|_{H^{s}} \leq \|K_5\|_{H^s} + \|T_{K_5}Y \|_{H^{s-\frac{1}{2}}} + \|T_{Y}K_5 \|_{H^{s}} +  \|\Pi(Y,K_5) \|_{H^{s}}\\
\lesssim& (1+\|Y\|_{L^\infty})\|K_5\|_{H^{s}} + \|K_5\|_{C^{-\frac{1}{2}}_{*}}\|Y\|_{H^{s+\frac{1}{2}}} \lesssim_\CalA \CalB \|\W\|_{H^{s+\frac{1}{2}}}.
\end{align*}
Therefore, we get the paralinearization of the capillary terms \eqref{CurvaturePara}.
\end{proof}

Gathering the paralinearization results for each part of the water wave system \eqref{WParaMat}, \eqref{RNonCapillary} and \eqref{CurvaturePara}, we end this subsection by stating that the water wave system can be reduced to a system of paradifferential equations.
\begin{proposition}[Paralinearization of the water waves] Let $s>0$, the water wave system \eqref{e:WR} can be rewritten as
 \begin{equation} \label{e:ParaWater}
\left\{
\begin{aligned}
 &  T_{D_t} \W +\partial_\alpha T_{(1+\W)(1-\bar{Y})}R = G\\
& T_{D_t} R  +i\gamma R +i\sigma T_{J^{-\frac{1}{2}}(1-Y)^2}\W_{\alpha \alpha}- 3i\sigma T_{J^{-\frac{1}{2}}(1-Y)^3 \W_\alpha} \W_\alpha - ig\W =  K, 
\end{aligned}
\right.
\end{equation}   
where $T_{D_t}: = \partial_t + T_{\ub}\partial_\alpha$ is the \textit{para-material derivative} and the source terms $(G,K)$ satisfy the bound
\begin{equation*}
 \|(G,K) \|_{\mathcal{H}^s} \lesssim_\CalA \CalB \|(\W,R) \|_{\mathcal{H}^s}.
\end{equation*}
\end{proposition}

\subsection{Reduction of the water wave system}
In this subsection, we work with Wahl\'{e}n variables to rewrite the paradifferential system \eqref{e:ParaWater}.
The use of Wahl\'{e}n variables was developed by Wahl\'{e}n \cite{MR2309783} to reformulate the water waves with constant vorticity in Zakharov-Craig-Sulem formulation.
Later it was applied in Berti-Franzoi-Maspero \cite{MR4228858,MR4658635} on the torus.  
Since we are working with the holomorphic coordinates, we derive the equivalent Wahlen variables and as the expected  new unknowns is slightly different.

Define new unknowns $(\eta, \zeta) : = (\W, R-i\frac{\gamma}{2}W)$, they are also holomorphic functions.
We then rewrite the paradifferential water wave system \eqref{e:ParaWater} using $(\eta, \zeta)$. 
\begin{proposition}
Let $s>0$, then  the paradifferential water wave system \eqref{e:ParaWater} can be written as
\begin{equation}
 \left\{
    \begin{array}{lr}
    T_{D_t} \eta  +i\frac{\gamma}{2}\eta +  T_{(1-\bar{Y})(1+\W)}\zeta_\alpha + T_{((1-\bar{Y})(1+\W))_\alpha}\zeta= \tilde{G}  &\\
    T_{D_t} \zeta +i\frac{\gamma}{2}\zeta + i\sigma T_{J^{-\frac{1}{2}}(1-Y)^2}\eta_{\alpha \alpha}- 3i\sigma T_{J^{-\frac{1}{2}}(1-Y)^3 \W_\alpha}\eta_\alpha-ig\eta-\frac{\gamma^2}{4}\partial_\alpha^{-1}\eta  = \tilde{K}, & 
             \end{array}
\right.  \label{e:etaxi}
\end{equation}
where  source terms $(\tilde{G},\tilde{K})$ satisfy
\begin{equation*}
 \|(\tilde{G},\tilde{K}) \|_{\mathcal{H}^s} \lesssim_\CalA \CalB \|(\W,R) \|_{\mathcal{H}^s}.
\end{equation*}
\end{proposition}

\begin{proof}
We substitute $(\W, R) = (\eta, \zeta + i\frac{\gamma}{2}\eta)$ in \eqref{e:ParaWater} and simplify the system.
Then we obtain the first equation of \eqref{e:etaxi} with 
\begin{equation*}
    \tilde{G} = G - i\frac{\gamma}{2}T_{\W-\bar{Y}-\W\bar{Y}}\W - i\frac{\gamma}{2}T_{(\W-\bar{Y}-\W\bar{Y})_\alpha} W,
\end{equation*}
which satisfies the source term bound by \eqref{HsLinfty}, \eqref{HsCmStar}.

For the second equation of \eqref{e:etaxi}, applying \eqref{WParaTwo} yields the result.
$\eta$ is the differentiated variable, and we write $\partial_\alpha^{-1}\eta$ here for $W$.
\end{proof}

Alternatively, the system \eqref{e:etaxi} can be rewritten as
\begin{equation*}
    T_{D_t} \begin{pmatrix} \eta \\
    \zeta \end{pmatrix} + 
    \begin{pmatrix}
    i\frac{\gamma}{2} &    T_\lambda \\
    T_k &  i\frac{\gamma}{2}
    \end{pmatrix}
    \begin{pmatrix} \eta \\
    \zeta \end{pmatrix}
    = \begin{pmatrix} \tilde{G} \\
    \tilde{K} \end{pmatrix},
\end{equation*}
where the symbols of the paradifferential operators are given by
\begin{equation}
\begin{aligned}
&\lambda(\alpha, \xi) = i(1-\bar{Y})(1+\W)\xi + ((1-\bar{Y})(1+\W))_\alpha, \\
& k(\alpha, \xi) = -i\sigma J^{-\frac{1}{2}}(1-Y)^2 \xi^2+ 3\sigma J^{-\frac{1}{2}}(1-Y)^3 \W_\alpha \xi -ig + i\frac{\gamma^2}{4\xi}.
\end{aligned} \label{LambdaKDef}   
\end{equation}
The use of  Wahl\'{e}n variables makes the diagonal components of the matrix equal; they are both $i\frac{\gamma}{2}$.

\subsection{Symmetrization of the water wave system}
In this subsection, we use the idea in Section 2 of Alazard, Baldi, and Han-Kwan \cite{MR3776276} to symmetrize the paradifferential system \eqref{e:ParaWater}.
Similar symmetrizations were carried out using Zakharov-Craig-Sulem formulation in \cite{MR2931520, MR3770970}.

To get a quick idea for the symmetrization, we first consider the linearized system of the water waves \eqref{e:etaxi} at zero solutions.
It is given by
\begin{equation}
 \left\{
    \begin{array}{lr}
    \partial_t \eta  +i\frac{\gamma}{2}\eta +  \zeta_\alpha = 0  &\\
    \partial_t \zeta +i\frac{\gamma}{2}\zeta + i\sigma \eta_{\alpha \alpha}-ig\eta-\frac{\gamma^2}{4}\partial_\alpha^{-1}\eta  = 0. & 
             \end{array}
\right.  \label{LinearZero}
\end{equation} 
Define the differential operators
\begin{equation}
    L(D) = \sqrt{\sigma|D|^3 + g|D|+ \frac{\gamma^2}{4}}, \quad M(D) = L(D)|D|^{-1}, \label{LDDef}
\end{equation}
then \eqref{LinearZero} can be symmetrize to the system
\begin{equation*}
    \partial_t \begin{pmatrix} M\eta \\
    \zeta\end{pmatrix} + 
    \begin{pmatrix}
    i\frac{\gamma}{2} &    -iL(D) \\
    -iL(D) &  i\frac{\gamma}{2}
    \end{pmatrix}
    \begin{pmatrix} M\eta \\
    \zeta \end{pmatrix}
    = \begin{pmatrix} 0 \\
    0 \end{pmatrix},
\end{equation*}
restricted to holomorphic functions.
The square matrix for dispersive terms is now a symmetric matrix whose two diagonal components are equal.

Inspired by the symmetrization of the linearized system at zero, for the nonlinear paradifferential system \eqref{e:etaxi}, we also want to symmetrize the system, but the computations should be carried out at the paradifferential level.
We aim to find symbols $p(\alpha, \xi), q(\alpha, \xi)$ such that new unknowns
\begin{equation*}
    U := T_p \eta, \quad V: = T_q \zeta
\end{equation*}
solve the paradifferential system
\begin{equation*}
    T_{D_t} \begin{pmatrix} U \\
    V \end{pmatrix} + 
    \begin{pmatrix}
    i\frac{\gamma}{2} &    -iL^{\frac{1}{2}}T_cL^{\frac{1}{2}} \\
    -iL^{\frac{1}{2}}T_cL^{\frac{1}{2}} &  i\frac{\gamma}{2}
    \end{pmatrix}
    \begin{pmatrix} U \\
    V \end{pmatrix}
    = \begin{pmatrix} G^\sharp \\
    K^\sharp \end{pmatrix},
\end{equation*}
for some symbol $c(\alpha, \xi)$ such that for $s>0$ the source terms satisfy  the bound
\begin{equation*}
 \|(G^\sharp,K^\sharp) \|_{H^s \times H^s} \lesssim_\CalA \CalB \|(\W,R) \|_{\mathcal{H}^s}.
\end{equation*}

The choices of the symbols $p$, $q$ and $c$ are not unique, they are allowed to be defined in an equivalent class modulo  admissible remainders.
To make the argument more precise, we give the following definition of the equivalent class of operators.

\begin{definition}
Let $s\in \mathbb{R}$,  and consider two families of operators of order $s$,
\begin{equation*}
    \{A(t): t\in [0,T] \}, \quad  \{B(t): t\in [0,T] \},
\end{equation*}
that depend on $(\W, R)$.
We write $A\sim B$  if $A-B$ is of order $s-\frac{3}{2}$, and the following condition is fulfilled: for all $\mu\in \mathbb{R}$,  such that for $a.e.$ $t\in [0,T]$,
\begin{equation*}
    \|A(t)-B(t)\|_{H^\mu \rightarrow H^{\mu -s + \frac{3}{2}}} \lesssim_\CalA \CalB.
\end{equation*}
\end{definition}

Let us consider two operators $T_a = T_{a^{(s_1)}+ a^{(s_1 -1)} }$ and $T_b = T_{b^{(s_2)}+ b^{(s_2 -1)} }$ such that
\begin{align*}
&a^{(s_1)} \in \Gamma^{s_1}_{\frac{3}{2}}, \quad a^{(s_1 -1)} \in \Gamma^{s_1 -1}_{\frac{1}{2}}, \quad b^{(s_2)} \in \Gamma^{s_2}_{\frac{3}{2}}, \quad b^{(s_2-1 )} \in \Gamma^{s_2-1 }_{\frac{1}{2}},\quad s_1 + s_2 =s,\\
& M^{s_1}_{\frac{3}{2}}(a^{(s_1)}) + M^{s_2}_{\frac{3}{2}}(b^{(s_2)}) \lesssim_\CalA \CalB, \quad  M^{s_1}_{0}(a^{(s_1)}) + M^{s_2}_{0}(b^{(s_2)}) \lesssim_\CalA 1, \\
 &M^{s_1-1}_{\frac{1}{2}}(a^{(s_1-1)}) + M^{s_2-1}_{\frac{1}{2}}(b^{(s_2-1)}) \lesssim_\CalA \CalB, \quad  M^{s_1-1}_{0}(a^{(s_1-1)}) + M^{s_2-1}_{0}(b^{(s_2-1)}) \lesssim_\CalA 1.
\end{align*}
Using the symbolic calculus \eqref{CompositionPara}, with $\rho = \frac{3}{2}, \frac{1}{2}$ respectively,
\begin{alignat*}{6}
&T_{a^{(s_1)}}T_{b^{(s_2)}} \sim T_{a^{(s_1)}b^{(s_2)}-i\partial_\xi a^{(s_1)}  \partial_\alpha b^{(s_2)}}, \quad &&T_{a^{(s_1)}}T_{b^{(s_2-1)}} \sim T_{a^{(s_1)} b^{(s_2-1)}}, \\
& T_{a^{(s_1-1)}}T_{b^{(s_2)}} \sim T_{a^{(s_1-1)} b^{(s_2)}}, \quad  &&T_{a^{(s_1-1)}}T_{b^{(s_2-1)}} \sim 0.
\end{alignat*}
Hence, we get
\begin{equation}
AB \sim T_{a^{(s_1)}b^{(s_2)}-i\partial_\xi a^{(s_1)}  \partial_\alpha b^{(s_2)} + a^{(s_1)} b^{(s_2-1)} + a^{(s_1-1)} b^{(s_2)}}. \label{SymbolicEquivalence}
\end{equation}

In order to symmetrize the paradifferential system \eqref{e:etaxi}, the matrix operator
\begin{equation*}
\begin{pmatrix}
T_p &    0 \\
0 &  T_q
\end{pmatrix}
\begin{pmatrix}
    i\frac{\gamma}{2} &    T_\lambda \\
    T_k &  i\frac{\gamma}{2}
\end{pmatrix}
-
\begin{pmatrix}
    i\frac{\gamma}{2} &    -iL^{\frac{1}{2}}T_cL^{\frac{1}{2}} \\
    -iL^{\frac{1}{2}}T_cL^{\frac{1}{2}} &  i\frac{\gamma}{2}
\end{pmatrix}
\begin{pmatrix}
T_p &    0 \\
0 &  T_q
\end{pmatrix}
\end{equation*}
needs to be a matrix of  admissible remainders.
Alternatively, it is equivalent to show the following result.
\begin{proposition} \label{t:TwoEquivalence}
For the operator $L(D)$ defined in \eqref{LDDef}, symbols $\lambda, k$ defined in \eqref{LambdaKDef}, and $\ell(\xi)$ be the symbol of the operator $L(D)$,  we consider the function $c(\alpha) := J^{-\frac{3}{4}}$ and  symbols
\begin{equation*}
    p(\alpha,\xi) = p^{(\frac{1}{2})}+ p^{(-\frac{1}{2})} , \quad q(\alpha,\xi) = q^{(0)}(\alpha)= J^{\frac{1}{4}},
\end{equation*}
where
\begin{equation*}
p^{(\frac{1}{2})} := -(1-Y)(1+\bar{\W})\xi^{-1}J^{-\frac{1}{2}}\ell(\xi),
\end{equation*}
and $p^{(-\frac{1}{2})}$ is defined in \eqref{PSubleading}.
Then  the following equivalence relations hold for the operators. 
\begin{equation}
iT_pT_\lambda \sim  L^{\frac{1}{2}}T_cL^{\frac{1}{2}} T_q, \quad iT_q T_k \sim  L^{\frac{1}{2}}T_cL^{\frac{1}{2}} T_p. \label{Equivalence}
\end{equation}   
\end{proposition}

\begin{proof}
Using symbolic calculus \eqref{CompositionPara},    
\begin{equation*}
L^{\frac{1}{2}}T_cL^{\frac{1}{2}} \sim T_m, \quad m = c\ell -\frac{i}{2}\partial_\xi \ell \partial_\alpha c.  
\end{equation*}
Hence, let $q = q^{(0)} + q^{(-1)}$, where $q^{(0)} \in \Gamma^0_{\frac{3}{2}}$ and $q^{(-1)} \in \Gamma^{-1}_{\frac{1}{2}}$, then
\begin{equation*}
L^{\frac{1}{2}}T_cL^{\frac{1}{2}}T_q \sim T_m T_q \sim T_{\tilde{\mathfrak{q}}_1}.   
\end{equation*}
Here the symbol $\tilde{\mathfrak{q}}_1$ equals
\begin{equation*}
\tilde{\mathfrak{q}}_1 = mq^{(0)}+ c\ell q^{(-1)} -ic \partial_\xi \ell \partial_\alpha q^{(0)} = c\ell q^{(0)} -\frac{i}{2}\partial_\xi \ell \partial_\alpha c q^{(0)} -ic \partial_\xi \ell \partial_\alpha q^{(0)}+ c\ell q^{(-1)}. 
\end{equation*}
Similarly, let $p = p^{(\frac{1}{2})}+ p^{(-\frac{1}{2})}$, where  $p^{(\frac{1}{2})} \in \Gamma^{\frac{1}{2}}_{\frac{3}{2}}$ and $p^{(-\frac{1}{2})} \in \Gamma^{-\frac{1}{2}}_{\frac{1}{2}}$, then
\begin{equation*}
L^{\frac{1}{2}}T_cL^{\frac{1}{2}}T_p \sim T_m T_p \sim T_{\tilde{\mathfrak{p}}_1}   
\end{equation*}
where the symbol $\tilde{\mathfrak{p}}_1$ equals
\begin{equation*}
\tilde{\mathfrak{p}}_1 = mp^{(\frac{1}{2})}+ c\ell p^{(-\frac{1}{2})} -ic \partial_\xi \ell \partial_\alpha p^{(\frac{1}{2})} = c\ell p^{(\frac{1}{2})} -\frac{i}{2}\partial_\xi \ell \partial_\alpha c p^{(\frac{1}{2})}  -ic \partial_\xi \ell \partial_\alpha p^{(\frac{1}{2})}+ c\ell p^{(-\frac{1}{2})}. 
\end{equation*}

The leading terms and the sub-leading terms in two equivalence relations need to match.
Hence, we essentially have four equations to solve.
However, there are five symbols $c$, $p^{(\frac{1}{2})}$, $p^{(-\frac{1}{2})}$, $q^{(0)}$ and $q^{(-1)}$ to determine.
Without loss of generality, we can assume that the lower order symbol $q^{(-1)} = 0$ so that the other four unknowns are solvable using four equations.

For the first equivalence relation in \eqref{Equivalence}, we compute the  symbol of $iT_pT_\lambda \sim T_{\tilde{\mathfrak{q}}_2}$ using \eqref{SymbolicEquivalence}:
\begin{align*}
\tilde{\mathfrak{q}}_2 = &-p^{(\frac{1}{2})}(1-\bar{Y})(1+\W)\xi + i\partial_\xi p^{(\frac{1}{2})}((1-\bar{Y})(1+\W))_\alpha \xi\\
&+ i p^{(\frac{1}{2})}((1-\bar{Y})(1+\W))_\alpha - p^{(-\frac{1}{2})}(1-\bar{Y})(1+\W)\xi.
\end{align*}
Since $T_{\tilde{\mathfrak{q}}_1}\sim T_{\tilde{\mathfrak{q}}_2}$, comparing the leading order terms gives
\begin{equation}
 p^{(\frac{1}{2})} = -(1-Y)(1+\bar{\W})\xi^{-1}c\ell q^{(0)}. \label{PLeading}
\end{equation}
The lower order terms give the expression for $p^{(-\frac{1}{2})}$,
\begin{equation}
 \begin{aligned}
   p^{(-\frac{1}{2})} &=   (1+\bar{\W})(1-Y)\xi^{-1}\Big(\frac{i}{2}\partial_\xi \ell \partial_\alpha c q^{(0)}\\
   &+ic \partial_\xi \ell \partial_\alpha q^{(0)}
+ i\partial_\xi p^{(\frac{1}{2})}((1-\bar{Y})(1+\W))_\alpha \xi+ i p^{(\frac{1}{2})}((1-\bar{Y})(1+\W))_\alpha \Big).
\end{aligned}   \label{PSubleading}
\end{equation}

For the second equivalence relation in \eqref{Equivalence},  we compute the equivalence symbol of $iT_qT_k \sim T_{\tilde{\mathfrak{p}}_2}$ using \eqref{SymbolicEquivalence}:
\begin{equation*}
 \tilde{\mathfrak{p}}_2 = \sigma J^{-\frac{1}{2}}(1-Y)^2 \xi^2 q^{(0)} + 3i\sigma J^{-\frac{1}{2}}(1-Y)^3 \W_\alpha \xi q^{(0)} +gJ^{-\frac{1}{2}}(1-Y)^2q^{(0)}-\frac{\gamma^2}{4\xi}J^{-\frac{1}{2}}(1-Y)^2q^{(0)},
\end{equation*}
since $q^{(0)}$ is of order zero and is independent of $\xi$. 
Comparing the leading terms of $\tilde{\mathfrak{p}}_1$ and $\tilde{\mathfrak{p}}_2$, one gets
\begin{equation*}
    -c^2 \ell^2(1-Y)(1+\bar{\W})q^{(0)}\xi^{-1} = J^{-\frac{1}{2}}(1-Y)^2 q^{(0)}\left(\sigma \xi^2 +g -\frac{\gamma^2}{4\xi}\right), \quad \xi<0.
\end{equation*}
Due to the fact that $-\ell^2(\xi) = \sigma \xi^3 +g\xi -\frac{\gamma^2}{4}$ for $\xi<0$, one can simplify this equation and get
\begin{equation*}
 c = J^{-\frac{3}{4}}.
\end{equation*}
For the sub-leading terms of $\tilde{\mathfrak{p}}_1$ and $\tilde{\mathfrak{p}}_2$,
\begin{equation*}
-\frac{i}{2}\partial_\xi \ell \partial_\alpha c p^{(\frac{1}{2})}  -ic \partial_\xi \ell \partial_\alpha p^{(\frac{1}{2})}+ c\ell p^{(-\frac{1}{2})} \approx 3i\sigma J^{-\frac{1}{2}}(1-Y)^3 \W_\alpha \xi q^{(0)}.
\end{equation*}
Here for two symbols $a\approx b$ we mean the the leading terms of $a$ and $b$ are the same.

To simplify the computation, we use the approximation $\ell(\xi)\approx |\sigma \xi^3 |^{\frac{1}{2}}$, so that
\begin{equation*}
\partial_\xi \ell \approx \frac{3}{2}\sqrt{\sigma}\frac{\xi}{\sqrt{|\xi|}}, \quad \partial_\alpha c = -\frac{3}{4}J^{-\frac{3}{4}}[(1-\bar{Y})\bar{\W}_\alpha + (1-Y)\W_\alpha],
\end{equation*}
and 
\begin{equation*}
p^{(\frac{1}{2})} \approx -(1-Y)(1+\bar{\W})J^{-\frac{3}{4}}\sqrt{\sigma}\frac{\xi}{\sqrt{|\xi|}} q^{(0)}.
\end{equation*}
Using also the formula for $p^{(-\frac{1}{2})}$ \eqref{PSubleading}, we obtain that
\begin{equation*}
    q^{(0)} = J^{\frac{1}{4}}.
\end{equation*}
Finally, the symbols $p^{(\frac{1}{2})}$ and $p^{(-\frac{1}{2})}$ are given in \eqref{PLeading} and \eqref{PSubleading}.
\end{proof}

Having obtained the equivalence relations for the operators \eqref{Equivalence}, we can now perform the symmetrization for the paradifferential system \eqref{e:etaxi}.
\begin{proposition}
 Let $s>0$, and $(\eta, \zeta)$ solve the paradifferential system \eqref{e:etaxi}, then for $c, p, q$ defined in Proposition \ref{t:TwoEquivalence}, the new unknowns
 \begin{equation*}
     U: = T_p \eta, \quad V: = T_q \zeta
 \end{equation*}
solve the symmetrized paradifferential system
\begin{equation}
 \left\{
    \begin{array}{lr}
    T_{D_t} U  +i\frac{\gamma}{2}U  -iL^{\frac{1}{2}}T_cL^{\frac{1}{2}}V= G^\sharp  &\\
    T_{D_t} V  -iL^{\frac{1}{2}}T_cL^{\frac{1}{2}}U + i\frac{\gamma}{2}V  = K^\sharp, & 
             \end{array}
\right.  \label{e:UVSym}
\end{equation}
where the source terms $(G^\sharp, K^\sharp)$  satisfy the bound
\begin{equation*}
 \|(G^\sharp,K^\sharp) \|_{H^s \times H^s} \lesssim_\CalA \CalB \|(\W,R) \|_{\mathcal{H}^s}.
\end{equation*}
\end{proposition}

\begin{proof}
Applying $T_p$ to the first equation of the system \eqref{e:etaxi} and $T_q$ to the second equation of the system \eqref{e:etaxi}, the new unknowns $(U ,V)$ solve the paradifferential system
\begin{equation*}
 \left\{
    \begin{array}{lr}
    T_{D_t} U  +i\frac{\gamma}{2}U  -iL^{\frac{1}{2}}T_cL^{\frac{1}{2}}V= T_p \tilde{G} + [T_\ub\partial_\alpha, T_p]\eta + i(iT_pT_\lambda - L^{\frac{1}{2}}T_cL^{\frac{1}{2}} T_q) \zeta + T_{\partial_t p} \eta &\\
    T_{D_t} V  -iL^{\frac{1}{2}}T_cL^{\frac{1}{2}}U + i\frac{\gamma}{2}V  = T_q \tilde{K} +[T_\ub\partial_\alpha, T_q]\zeta +i(iT_qT_k - L^{\frac{1}{2}}T_cL^{\frac{1}{2}} T_p) \eta + T_{\partial_t q} \zeta. & 
             \end{array}
\right.  
\end{equation*}
For the first terms on the right-hand side of the above system,
\begin{equation*}
\|T_p \tilde{G}\|_{H^s} + \| T_q \tilde{K}\|_{H^s} \lesssim \|\tilde{G}\|_{H^{s+\frac{1}{2}}} + \|\tilde{K}\|_{H^s} \lesssim_\CalA \CalB \|(\W,R) \|_{\mathcal{H}^s}.
\end{equation*}
For the commutator terms, using the composition rule \eqref{CompositionPara} in paradifferential symbolic calculus,
\begin{equation*}
\|[T_\ub\partial_\alpha, T_p]\eta\|_{H^s}+ \|[T_\ub\partial_\alpha, T_q]\zeta  \|_{H^s}\lesssim_\CalA \CalB \|(\eta, \zeta) \|_{\mathcal{H}^s} \lesssim_\CalA \CalB \|(\W,R) \|_{\mathcal{H}^s}.
\end{equation*}
Then owing to the equivalence relations \eqref{Equivalence} which were established in Proposition \ref{t:TwoEquivalence},
$iT_pT_\lambda - L^{\frac{1}{2}}T_cL^{\frac{1}{2}} T_q$ is of order $0$, and $iT_qT_k - L^{\frac{1}{2}}T_cL^{\frac{1}{2}} T_p$ is of order $\frac{1}{2}$ so that
\begin{equation*}
\|(iT_pT_\lambda - L^{\frac{1}{2}}T_cL^{\frac{1}{2}} T_q) \zeta \|_{H^s} + \|(iT_qT_k - L^{\frac{1}{2}}T_cL^{\frac{1}{2}} T_p) \eta \|_{H^s} \lesssim_\CalA \CalB \|(\W,R) \|_{\mathcal{H}^s}.
\end{equation*}
It suffices to estimate the last terms on the right-hand side of the above system, which is further reduced to showing the operator bounds
\begin{equation*}
    \|T_{\partial_t p} \|_{H^{s+\frac{1}{2}} \rightarrow H^s} +  \|T_{\partial_t q} \|_{H^s \rightarrow H^s} \lesssim_\CalA \CalB.
\end{equation*}
Note that $p$ and $q$ both depend on $\W$ and $\bar{\W}$. 
Recall from the paradifferential equation \eqref{WParaMat}, $\|\partial_t \W\|_{L^\infty} \lesssim_\CalA \CalB$.
Hence,
\begin{equation*}
M^{\frac{1}{2}}_0(\partial_t p^{(\frac{1}{2})})+ M_0^0(\partial_t q)\lesssim_\CalA \CalB.   
\end{equation*}
To estimate $\| T_{\partial_t p^{(-\frac{1}{2})}}\|_{H^{s+\frac{1}{2}}\rightarrow H^s}$, we apply Lemma \ref{t:NegativeRhoCal}.
It follows from the expression \eqref{PSubleading} that
\begin{align*}
\partial_t p^{(-\frac{1}{2})} = \sum_k& F_k^1(\W, \bar{\W}, \W_\alpha, \bar{\W}_\alpha, \xi)\partial_t \W_\alpha + F_k^2(\W, \bar{\W}, \W_\alpha, \bar{\W}_\alpha,
\xi)\partial_t \W \\
+&  F_k^3(\W, \bar{\W}, \W_\alpha, \bar{\W}_\alpha, \xi)\partial_t \bar{\W}_\alpha + F_k^4(\W, \bar{\W}, \W_\alpha, \bar{\W}_\alpha, \xi)\partial_t \bar{\W},
\end{align*}
for some functions $F_k^i$, $i=1,2,3,4$. 
Since $\| \partial_t \W_\alpha\|_{C^{-1}_*} \lesssim_\CalA \CalB$, and $M_0^{-\frac{1}{2}}(F_k^i)\lesssim_\CalA 1$, by using the product estimate \eqref{CNegativeAlpha},
\begin{equation*}
    M^{-\frac{1}{2}}_{-1}(\partial_t p^{(-\frac{1}{2})}) \lesssim_\CalA \CalB.
\end{equation*}
Therefore, we get
\begin{equation*}
    \|T_{\partial_t p}\eta\|_{H^s}+ \|T_{\partial_t q}\zeta \|_{H^s}\lesssim_\CalA \CalB \| (\W, R)\|_{H^s}.
\end{equation*}
Finally, collecting all source terms on the right-hand side into $(G^\sharp, K^\sharp)$,
\begin{align*}
&G^\sharp = T_p \tilde{G} + [T_\ub\partial_\alpha, T_p]\eta + i(iT_pT_\lambda - L^{\frac{1}{2}}T_cL^{\frac{1}{2}} T_q) \zeta + T_{\partial_t p} \eta, \\
&K^\sharp = T_q \tilde{K} +[T_\ub\partial_\alpha, T_q]\zeta +i(iT_qT_k - L^{\frac{1}{2}}T_cL^{\frac{1}{2}} T_p) \eta + T_{\partial_t q} \zeta,
\end{align*}
using the estimates above, then $(G^\sharp, K^\sharp)$ satisfy the desired bounds.
\end{proof}

At the end of this section, we rewrite the symmetrized paradifferential system \eqref{e:UVSym} as a single paradifferential equation.
Taking the imaginary part of the first equation in \eqref{e:UVSym} and the real part of the second
equation in \eqref{e:UVSym}, using Lemma \ref{t:RealIm} and the fact that $\ub$ and $c$ are real-valued functions, one obtains the system
\begin{equation*}
 \left\{
    \begin{array}{lr}
    T_{D_t} \Im U  +\frac{\gamma}{2} \Re U -L^{\frac{1}{2}}T_cL^{\frac{1}{2}} \Re  V = \Im G^\sharp  &\\
    T_{D_t} \Re V   + L^{\frac{1}{2}}T_cL^{\frac{1}{2}} \Im U -\frac{\gamma}{2}\Im V = \Re K^\sharp. & 
             \end{array}
\right. 
\end{equation*}
Since $U$ and $V$ are both holomorphic functions, it follows that their real parts and imaginary parts can be related using the Hilbert transform:
\begin{equation*}
    \Re U = H \Im U, \quad \Im V = -H \Re V. 
\end{equation*}
This system can be further reduced to a single paradifferential equation by introducing the complex-valued unknown $\Phi = \Re V + i\Im U$.
As a consequence, we rewrite above system using $\Re V$ and $\Im U$ as
\begin{equation*}
 \left\{
    \begin{array}{lr}
    T_{D_t} \Im U  +\frac{\gamma}{2}H \Im U -L^{\frac{1}{2}}T_cL^{\frac{1}{2}} \Re  V = \Im G^\sharp  &\\
    T_{D_t} \Re V   + L^{\frac{1}{2}}T_cL^{\frac{1}{2}} \Im U + \frac{\gamma}{2}H\Re V = \Re K^\sharp. & 
             \end{array}
\right. 
\end{equation*}
To sum up, we arrive at the final reduction result for the water wave system in this section.
\begin{proposition} \label{t:SingleParaEqn}
Let $s>0$, and $(\W, R)$ solve the paradifferential water wave system \eqref{e:ParaWater}. 
Let $p,q,\ell,c$ be defined in Proposition \ref{t:TwoEquivalence}, then the unknown
\begin{equation*}
    \Phi : = \Re T_q\left(R-i\frac{\gamma}{2}W\right) + i\Im T_p \W
\end{equation*}
solves the paradifferential equation
\begin{equation}
    T_{D_t}\Phi -i\left(L^{\frac{1}{2}}T_c L^{\frac{1}{2}} +\frac{\gamma}{2}\text{sgn}(D) \right)\Phi  = F, \label{e:SinglePara}
\end{equation}
where the source term $F: = \Im G^\sharp + i\Re K^\sharp$ satisfies
\begin{equation}
    \|F\|_{H^s} \lesssim_\CalA \CalB \| (\W, R)\|_{\mathcal{H}^s}. \label{FHsBound}
\end{equation}
\end{proposition}

\section{A priori energy estimate and the Strichartz estimate} \label{s:Estimate}
In this section, we first prove an a priori energy estimate Theorem \ref{t:EnergyEstimate} for the paradifferential equation \eqref{e:SinglePara}, then we use the results in Nguyen \cite{MR3724757} to establish the Strichartz estimate for the paradifferential equation \eqref{e:SinglePara}.

\subsection{Energy estimate} \label{s:Energy}
To obtain an $H^s$ estimate, we apply an elliptic operator of order $s$ to the paradifferential equation \eqref{e:SinglePara}.
To avoid the issue of the commutator with the dispersive term, here we do not use the usual elliptic differential operator $(1+D^2)^{\frac{s}{2}}$.
Instead, we apply the elliptic operator $T_\wp$ of order $s$ as in \cite{MR2805065, MR3770970}, where $\wp = (c\ell )^{\frac{2s}{3}}$, and $c\ell$ is the principle symbol  of the operator $L^{\frac{1}{2}}cL^{\frac{1}{2}}$.

Taking $\varphi := T_\wp \Phi$, then by \eqref{TABound}, 
\begin{equation}
    \|\varphi\|_{L^2}\lesssim \| \Phi \|_{H^s} \lesssim_\CalA \|(\W, R) \|_{\mathcal{H}^s}. \label{varphiLTwo}
\end{equation}
On the other hand, using Proposition $3.13$ in \cite{MR3770970}, for any  constant $\mu\geq 0$,
\begin{equation} 
\begin{aligned}
\|\Phi \|_{H^s} \lesssim \|(\W, R) \|_{\mathcal{H}^s} &\lesssim \|(\eta, \zeta) \|_{H^s} \lesssim_\CalA \|T_\wp T_p \eta \|_{L^2} + \|T_\wp T_q \zeta  \|_{L^2} + \| (\eta, \zeta)\|_{\mathcal{H}^{-\mu}} \\
&\lesssim_\CalA \|\varphi\|_{L^2} + \| (\eta, \zeta)\|_{\mathcal{H}^{-\mu}} \lesssim_\CalA \|\varphi\|_{L^2} + \| (\W, R)\|_{\mathcal{H}^{-\mu}}.   
\end{aligned} \label{WRMuMinus}
\end{equation}
The unknown $\varphi$ solves the equation
\begin{equation}
 T_{D_t}\varphi -i\left(L^{\frac{1}{2}}T_c L^{\frac{1}{2}} +\frac{\gamma}{2}\text{sgn}(D) \right)\varphi  = T_\wp F+ \tilde{f}, \label{e:SingleParaSRegularity}
\end{equation}
where $\tilde{f}$ is given by
\begin{equation*}
\tilde{f} = T_{\partial_t \wp}\Phi + [T_\ub \partial_\alpha, T_\wp]\Phi - i\left[ L^{\frac{1}{2}}T_c L^{\frac{1}{2}} +\frac{\gamma}{2}\text{sgn}(D) , T_\wp\right ]\Phi.    
\end{equation*}
We can derive the following a priori energy estimate for the equation \eqref{e:SingleParaSRegularity}.
\begin{lemma}
Suppose  $s>0$,  $\Phi$ solves the paradifferential equation \eqref{e:SinglePara}.
Then for any $\mu \geq 0$, $\varphi = T_\wp \Phi$ with $\wp = (c\ell )^{\frac{2s}{3}}$ satisfies the a priori energy estimate
\begin{equation}
    \frac{d}{dt}\| \varphi\|_{L^2}^2 \lesssim_\CalA (1+\CalB) \left(\|(\W, R) \|_{\mathcal{H}^{-\mu}} + \|\varphi\|_{L^2}\right) \|\varphi\|_{L^2}. \label{VarphiTimeDerivative}
\end{equation}
\end{lemma}

\begin{proof}
$\varphi$ solves the equation \eqref{e:SingleParaSRegularity}, and we estimate the right-hand side source terms in $L^2$.
Using \eqref{FHsBound} and \eqref{WRMuMinus}, we get
\begin{equation*}
\|T_\wp F \|_{L^2} \lesssim_\CalA \CalB \|(\W, R)\|_{\mathcal{H}^s} \lesssim_\CalA \CalB \left(\|\varphi\|_{L^2} + \| (\W, R)\|_{\mathcal{H}^{-\mu}}\right).
\end{equation*}
For the first term of $\tilde{f}$, we note that $\partial_t \wp = F_1(\W, \bar{\W}, \xi) \partial_t \W + F_2(\W, \bar{\W}, \xi) \partial_t \bar{\W}$ for some functions $F_1, F_2$.
Since $\|\partial_t \W\|_{L^\infty} \lesssim_\CalA \CalB$, it follows that $T_{\partial_t \wp}$ is of order $s$, and
\begin{equation*}
\|T_{\partial_t \wp}\Phi \|_{L^2} \lesssim_\CalA \CalB \| \Phi\|_{H^s}.
\end{equation*}
Next, we estimate two commutator terms in $\tilde{f}$.
For the commutator with the transport term, we use the symbolic calculus \eqref{CompositionPara} with $\rho =1$,
\begin{equation*}
    \| [T_\ub \partial_\alpha, T_\wp]\Phi\|_{L^2} \lesssim_\CalA (1+\CalB ) \|\Phi \|_{H^s}.
\end{equation*}
As for the commutator with the dispersive term, since $\wp$ is some power of $c\ell$, $\partial_\xi \wp \partial_\alpha (c\ell) = \partial_\alpha \wp \partial_\xi (c\ell)$.
Using again the symbolic calculus \eqref{CompositionPara} with $\rho = \frac{3}{2}$, the leading term of the commutator $[L^{\frac{1}{2}}T_c L^{\frac{1}{2}}, T_\wp]\Phi$ vanishes, and
\begin{equation*}
\left\| i\left[ L^{\frac{1}{2}}T_c L^{\frac{1}{2}} +\frac{\gamma}{2}\text{sgn}(D) , T_\wp\right ]\Phi \right\|_{L^2} \lesssim_\CalA (1+\CalB ) \|\Phi \|_{H^s}.
\end{equation*}
Putting these estimates together, and using \eqref{TABound}, \eqref{WRMuMinus} we obtain the estimate
\begin{equation*}
\|\tilde{f} \|_{L^2} \lesssim_\CalA (1+\CalB ) \| \Phi\|_{H^s}\lesssim_\CalA  (1+ \CalB ) \left(\|\varphi\|_{L^2}+ \| (\W, R)\|_{\mathcal{H}^{-\mu}}\right).
\end{equation*}
Then we compute 
\begin{equation*}
\frac{d}{dt}\|\varphi \|^2_{L^2} = 2\Re \langle \varphi_t, \varphi\rangle = \langle (\partial_\alpha (T_\ub)^*  - T_\ub \partial_\alpha) \varphi + i L^{\frac{1}{2}}(T_c - (T_c)^*) L^{\frac{1}{2}} \varphi, \varphi \rangle + 2\Re \langle T_\wp F+ \tilde{f}, \varphi \rangle .
\end{equation*}
Using the paradifferential calculus for adjoint operator \eqref{AdjointBound}, 
\begin{align*}
\frac{d}{dt}\|\varphi \|^2_{L^2} &\lesssim \| (\partial_\alpha (T_\ub)^*  - T_\ub \partial_\alpha) \varphi + i L^{\frac{1}{2}}(T_c - (T_c)^*) L^{\frac{1}{2}} \varphi\|_{L^2} \| \varphi\|_{L^2} + (\|T_\wp F \|_{L^2} + \|\tilde{f}\|_{L^2}) \|\varphi\|_{L^2}\\
& \lesssim_\CalA (1+ \CalB ) \left(\|(\W, R) \|_{\mathcal{H}^{-\mu}} + \|\varphi\|_{L^2}\right) \|\varphi\|_{L^2}.
\end{align*}
This gives the a priori energy estimate for $\varphi$.
\end{proof}
Having obtained the a priori energy estimate for $\varphi$, we can now deduce the energy estimate for $(\W, R)$ and finish the proof of Therorem \ref{t:EnergyEstimate}.

\begin{proof}[Proof of Therorem \ref{t:EnergyEstimate}]
Integrating the a priori energy estimate inequality \eqref{VarphiTimeDerivative}, and using \eqref{varphiLTwo}, \eqref{WRMuMinus}, we get the Sobolev bound
\begin{align*}
    &\|(\W, R)(t) \|_{\mathcal{H}^s}^2 \lesssim_\CalA \|(\W, R) (t)\|^2_{\mathcal{H}^{-\mu}} + \| \varphi(t)\|_{L^2} \\
    \lesssim_\CalA & \|(\W, R) (t)\|^2_{\mathcal{H}^{-\mu}} +  \|(\W, R)(0) \|_{\mathcal{H}^s}^2 + \int_0^t (1+ \CalB(\tau))  \|(\W, R)(\tau)  \|_{\mathcal{H}^s}^2  d\tau.
\end{align*}
Note that by choosing appropriate $\mu \geq 0$, one can obtain the energy estimate directly for the system \eqref{e:WR},
\begin{equation*}
    \frac{d}{dt}\|(\W, R) (t)\|^2_{\mathcal{H}^{-\mu}} \lesssim_\CalA (1+\CalB ) \|(\W, R) (t)\|^2_{\mathcal{H}^{s}},
\end{equation*}
so that integrating it gives
\begin{equation*}
 \|(\W, R) (t)\|^2_{\mathcal{H}^{-\mu}} \lesssim_\CalA \|(\W, R) (0)\|^2_{\mathcal{H}^{-\mu}}     + \int_0^t (1+ \CalB(\tau) ) \|(\W, R)(\tau)  \|_{\mathcal{H}^s}^2  d\tau.
\end{equation*}
Adding this energy bound to the previous inequality, 
\begin{equation*}
 \|(\W, R)(t) \|_{\mathcal{H}^s}^2 \lesssim_\CalA   \|(\W, R)(0) \|_{\mathcal{H}^s}^2 + \int_0^t (1 +\CalB(\tau) )  \|(\W, R)(\tau)  \|_{\mathcal{H}^s}^2  d\tau, 
\end{equation*}
which is exactly the integral version of the energy estimate \eqref{EstimateWREnergy}.
\end{proof}

\subsection{The Strichartz estimate} \label{s:Strichartz}
To apply the Strichartz estimate for the paradifferential equation \eqref{e:SinglePara}, we first use the technique of \textit{para-composition}  of Alinhac \cite{MR0814548} to rewrite the equation, so that its principle dispersive part becomes constant-coefficient after the transformation.
In other words, we seek  an operator $\kappa^*$ such that $\tilde{\Phi}: = \kappa^* \Phi$ solves an equation of the following type:
\begin{equation*}
\partial_t \tilde{\Phi} + T_{\tilde{\ub}} \partial_\alpha \tilde{\Phi} - i|D|^{\frac{3}{2}} \tilde{\Phi} = \tilde{F}.
\end{equation*}
We remark that this method was previously used in Alazard-Burq-Zuily \cite{MR2931520} and Nguyen \cite{MR3724757} on the Strichartz estimate of two-dimensional gravity-capillary water waves with zero vorticity.
In the following, we recall the main result in Section 3 of \cite{MR3724757} on the global para-composition.

\begin{theorem}[Global para-composition \cite{MR3724757}] \label{t:ParaComposition}
Let $\kappa: \mathbb{R} \rightarrow \mathbb{R}$ be a diffeomorphism that satisfies the following two conditions:
\begin{enumerate}
\item There exist $\rho>0$, $r>-1$, $k\in \mathbb{N}$ such that $\partial_x \kappa \in C_*^\rho(\mathbb{R})$, $\partial_x^k \kappa \in H^{r+1-k}(\mathbb{R})$.
\item There exists a constant $m_0>0$ such that $\forall x\in \mathbb{R}$, $|\kappa^{'}(x)|\geq m_0$.
\end{enumerate}
Then there exists an operator $\kappa^{*}_g$ defined by
\begin{equation*}
 \kappa^{*}_g u = \sum_{k=0}^\infty P_k[ P_k u \circ \kappa], \quad \text{where } (u\circ \kappa) ( \alpha) = u(\kappa( \alpha))
\end{equation*}
with the following properties:
\begin{enumerate}
\item Operation: For every $s\in \mathbb{R}$,
\begin{align*}
    &\|\kappa^{*}_g u \|_{C_{*}^s} \lesssim_{m_0, \|\kappa^{'} \|_{L^\infty}}  \|u \|_{C_{*}^s}, \quad \forall u\in C_{*}^s(\mathbb{R}), \\
    &\|\kappa^{*}_g u \|_{H^s} \lesssim_{m_0, \|\kappa^{'} \|_{L^\infty}}  \|u \|_{H^s}, \quad \forall u\in H^s(\mathbb{R}).
\end{align*}
\item Conjugation: Let $m,s\in \mathbb{R}$ and $\tau>0$. 
Set $\epsilon = \min \{\tau, \rho \}$.
Then for every $h(x, \xi)\in \Gamma^m_\tau$, homogeneous in $\xi$ there exists $h^* \in \Gamma^m_\epsilon$, such that  for all $u\in H^s(\mathbb{R})$, 
\begin{equation*}
    \kappa^{*}_g T_h u = T_{h^*}\kappa^{*}_g u + R_{conj}u,
\end{equation*}
where the remainder term $R_{conj}u$ satisfies
\begin{equation*}
    \|R_{conj}u \|_{H^{s-m+\epsilon}} \lesssim_{m_0, \|\kappa^{'} \|_{C^\rho_{*}}} M^m_\tau(h) (1+\|\partial^{k}_x \kappa \|_{H^{r+1 -k}}) \|u\|_{H^s}.
\end{equation*}
Moreover, $h^{*}(x, \xi)$ is given by the formula
\begin{align*}
&h^{*}(x, \xi) = \sum_{j=0}^{[\rho]}h_j^{*}:=  \sum_{j=0}^{[\rho]}\frac{1}{j!}\partial_\xi^j D_y^j\left(h(\kappa(x), R(x,y)^{-1}\xi)\frac{|\partial_y \kappa(y)|}{|R(x,y)|} \right)\Bigg|_{y= x}, \\
&R(x, y) = \int_0^1 \partial_x \kappa\left(tx +(1-t)y \right) dt.
\end{align*}
\item Linearization: The linearization of $\kappa^{*}_g u$ is given by $u\circ \kappa - \dot{T}_{u^{'}\circ \kappa}\kappa$ and the remainder part
\begin{equation*}
    R_{line} u: = \kappa_g^{*}u -u\circ \kappa + \dot{T}_{u^{'}\circ \kappa}\kappa
\end{equation*}
satisfies the estimates
\begin{enumerate}
\item If $0<\sigma<1$, $\rho+ \sigma >1$, and $r+\sigma >0$, then for $\tilde{s} = \min \{s+\rho, r+\sigma \}$,
\begin{equation*}
\| R_{line} u \|_{H^{\tilde{s}}} \lesssim_{m_0, \| \kappa^{'}\|_{C^\rho_{*}}}(1+ \|\partial^k_x \kappa \|_{H^{r+1-k}})(\|\partial_x u \|_{H^{s-1}}+ \|u\|_{C^\sigma_{*}}).
\end{equation*}
\item If $\sigma>1$, set $\epsilon = \min (\sigma-1, \rho+1)$, then for $\tilde{s} = \min \{s+\rho, r+1+\epsilon \}$,
\begin{equation*}
\| R_{line} u \|_{H^{\tilde{s}}} \lesssim_{m_0, \| \kappa^{'}\|_{C^\rho_{*}}}(1+ \|\partial^k_x \kappa \|_{H^{r+1-k}})(\|\partial_x u \|_{H^{s-1}}+ \|u\|_{C^\sigma_{*}}).
\end{equation*}
\end{enumerate}
\end{enumerate}
\end{theorem}
The expression of $\kappa^{*}_g$ is not important for our analysis, we will instead let $\kappa^*$ be the linearization of $\kappa^{*}_g$, and we define
\begin{equation*}
 \tilde{\Phi} :=  \kappa^* \Phi = \Phi\circ \kappa - \dot{T}_{\partial_\alpha \Phi\circ \kappa}\kappa, \quad \text{where } \Phi \circ \kappa (t, \alpha) = \Phi(t, \kappa(t, \alpha)) 
\end{equation*}
For each $t\in [0, T]$, we define a time dependent $\mathbb{R} \mapsto \mathbb{R}$  diffeomorphism $\chi(t, \alpha)$ by
\begin{equation*}
\chi(t, \alpha) = \int_0^\alpha J^{\frac{1}{2}}(t,y) \,dy.
\end{equation*}
This choice of $\chi(t, \alpha)$ gives $\partial_\alpha \chi = c^{-\frac{2}{3}}(t, \alpha)$.
We further set $\kappa$ to be the inverse of $\chi$, then
\begin{equation*}
\partial_\alpha \kappa(t, \alpha) = \frac{1}{(\partial_\alpha \chi)\circ \kappa} = J^{-\frac{1}{2}}\circ \kappa,
\end{equation*}
and 
\begin{equation*}
    1\geq \partial_\alpha \kappa \geq m_0 : = (1+ \|\W \|_{L^\infty_{t,\alpha}})^{-\frac{1}{2}}.
\end{equation*}
Using the Moser type estimate \eqref{MoserOne} and Lemma \ref{t:Composition}, for $s\geq 0$,
\begin{equation}
 \|\partial_\alpha \chi -1 \|_{H^{s}} \lesssim \|\W \|_{H^{s}}, \quad \|\partial_\alpha \kappa -1 \|_{H^{s}} \lesssim \|\W \|_{H^{s}} \label{HsChiAlpha}
\end{equation}
To get an estimate for $\partial_t \chi$, we write
\begin{equation*}
    \partial_t \chi(t, \alpha) = \int_0^\alpha \frac{1}{2}J^{-\frac{1}{2}}(t,y)(1+\bar{\W}(t,y))\W_t(t,y)+ \frac{1}{2}J^{-\frac{1}{2}}(t,y)(1+\W(t,y))\bar{\W}_t(t,y) \, dy. 
\end{equation*}
Using the embedding $\dot{H}^1(\mathbb{R})\hookrightarrow C^{\epsilon}_{*}(\mathbb{R})$, for $s\geq 1$,
\begin{equation}
    \| \partial_t \chi(t, \alpha) \|_{C^{\epsilon}_{*}} +  \| \partial_t \partial_\alpha \chi(t, \alpha) \|_{L^2} \lesssim \| J^{-\frac{1}{2}}(1+\bar{\W})\W_t\|_{L^2} \lesssim_{\mathcal{A}} \| \W_t\|_{L^2} \lesssim \| (\W ,R)(t)\|_{\mathcal{H}^s}.  \label{TimeDerivativeChi}
\end{equation}
In addition, we have
\begin{equation}
 \| \partial_t \partial_\alpha \chi(t, \alpha) \|_{C^{\epsilon}_{*}}\lesssim \|J^{-\frac{1}{2}}(1+\bar{\W})\W_t \|_{C^{\epsilon}_{*}}\lesssim_\CalA \|(\W,R)(t)\|_{\mathcal{W}^r}, \quad r>1.  \label{TimeDerivativeChiTwo}
\end{equation}
Since $\chi$ is the inverse diffeomorphism of $\kappa$, $\chi(t, \kappa(t, \alpha)) = \alpha$, taking the time-derivative gives
\begin{equation*}
    \partial_t \chi \circ \kappa + \partial_\alpha \chi \circ \kappa \cdot \partial_t \kappa =0.
\end{equation*}
Hence, using \eqref{TimeDerivativeChi} and \eqref{TimeDerivativeChiTwo}, for $s\geq 1$, $r>1$,
\begin{align}
&\|\partial_t \kappa(t, \alpha)\|_{L^\infty} = \left\|\frac{\partial_t \chi}{\partial_\alpha \chi}\circ \kappa \right\|_{L^\infty} \leq  m_0\|\partial_t \chi \circ \kappa \|_{L^\infty} \lesssim \| (\W ,R)(t)\|_{\mathcal{H}^s},\label{TimeDerivativeKappa} \\
& \|\partial_t \kappa(t, \alpha) \|_{C^{1+\epsilon}_{*}} \lesssim \|\partial_t \chi\circ \kappa \|_{C^{1+\epsilon}_{*}} \|\partial_\alpha \kappa(t, \alpha) \|_{C^{1+\epsilon}_{*}} \lesssim_\CalA \|(\W,R)(t)\|_{\mathcal{W}^r}.  \label{TimeDerivativeKappaTwo}
\end{align}

Having established all the necessary bounds, we now apply Theorem \ref{t:ParaComposition} to show that $\tilde{\Phi} = \kappa^{*}\Phi$ solves a dispersive equation with its highest order term being constant coefficient.
\begin{proposition} \label{t:ConstantDispersive}
Let $\Phi$ solves  the paradifferential equation \eqref{e:SinglePara}, then for $s>r> 1$,  $\tilde{\Phi} = \kappa^{*}\Phi$ solves the paradifferential equation
\begin{equation}
 \left( \partial_t + T_{\tilde{\ub}}\partial_\alpha - i|D|^{\frac{3}{2}}\right)\tilde{\Phi} = \tilde{F}, \label{ConstantDispersive}
\end{equation}
where the new para-coefficient for the transport term is
\begin{equation*}
    \tilde{\ub} = (\ub \circ \kappa)(\partial_\alpha \chi \circ \kappa) + \partial_t \chi \circ \kappa,
\end{equation*}
and the new source term $\tilde{F}$ satisfies, for $t\in [0,T]$,
\begin{equation*}
\|\tilde{F}(t)\|_{H^{s-\frac{1}{2}}} \lesssim_{\|(\W, R) \|_{L^\infty_t \mathcal{H}^s_\alpha}} 1+\|(\W,R)(t)\|_{\mathcal{W}^{r}}.
\end{equation*}
\end{proposition}
\begin{proof}
We apply the operator $\kappa^{*}$ to equation \eqref{e:SinglePara}.
First, for the term of time derivative, we use the definition of $\kappa^{*}$,
\begin{align*}
&\kappa^{*}(\partial_t \Phi) = \partial_t \Phi \circ \kappa -\dot{T}_{(\partial_t \partial_\alpha \Phi)\circ \kappa}\kappa = \partial_t(\Phi \circ \kappa ) - (\partial_\alpha \Phi \circ \kappa)\partial_t \kappa - \dot{T}_{(\partial_t \partial_\alpha \Phi)\circ \kappa}\kappa\\
=& \partial_t(\kappa^{*}\Phi) +\dot{T}_{(\partial_\alpha^2 \Phi \circ \kappa)\partial_t \kappa }\kappa+ \left(\dot{T}_{(\partial_\alpha \Phi)\circ \kappa}\partial_t \kappa - (\partial_\alpha \Phi \circ \kappa)\partial_t \kappa \right).
\end{align*}
Note that $\dot{T}_a u$ does not involve the low frequency part of $u$.
We estimate using \eqref{HsChiAlpha}, \eqref{TimeDerivativeKappaTwo}, \eqref{HsCmStar}, and \eqref{CNegativeAlpha},
\begin{align*}
 &\|\dot{T}_{(\partial_\alpha^2 \Phi \circ \kappa)\partial_t \kappa }\kappa \|_{H^{s+\frac{1}{2}}} \lesssim \|(\partial_\alpha^2 \Phi \circ \kappa)\partial_t \kappa\|_{C^{-1}_{*}}\|\partial_\alpha^2 \kappa \|_{H^{s-\frac{1}{2}}} \lesssim \|\partial_\alpha \Phi \circ \kappa\|_{C^0_{*}} \|\partial_t \kappa \|_{C^{1+\epsilon}_{*}}  \|\partial_\alpha^2 \kappa \|_{H^{s-\frac{1}{2}}} \\
 \lesssim& \|(\W, R) \|_{\mathcal{H}^s}\|(\W,R)\|_{\mathcal{W}^r}^2.
\end{align*}
To estimate the last term of $\kappa^{*}(\partial_t \Phi)$, we use the fact that
\begin{equation*}
    \partial_\alpha (\kappa^{*}\Phi) = (\partial_\alpha \Phi)\circ \kappa\cdot \partial_\alpha \kappa -\dot{T}_{ (\partial_\alpha \Phi)\circ \kappa}\partial_\alpha \kappa - \dot{T}_{\partial_\alpha((\partial_\alpha \Phi)\circ \kappa)} \kappa,  
\end{equation*}
so that
\begin{align*}
&\dot{T}_{(\partial_\alpha \Phi)\circ \kappa}\partial_t \kappa - (\partial_\alpha \Phi \circ \kappa)\partial_t \kappa = - \dot{T}_{\partial_t\kappa} (\partial_\alpha \Phi)\circ \kappa - \dot{\Pi}((\partial_\alpha \Phi)\circ \kappa, \partial_t\kappa)\\
 =& \dot{T}_{(\partial_t \chi)\circ \kappa \cdot \partial_\alpha \kappa} (\partial_\alpha \Phi)\circ \kappa + \dot{\Pi}((\partial_\alpha \Phi)\circ \kappa, (\partial_t \chi)\circ \kappa \cdot \partial_\alpha \kappa)\\
 =& \dot{T}_{(\partial_t \chi)\circ \kappa}\dot{T}_{\partial_\alpha \kappa} (\partial_\alpha \Phi)\circ \kappa + \dot{\Pi}((\partial_\alpha \Phi)\circ \kappa, (\partial_t \chi)\circ \kappa \cdot \partial_\alpha \kappa) +  F_{para}\\
 =& \dot{T}_{(\partial_t \chi)\circ \kappa}\left((\partial_\alpha \Phi)\circ \kappa\cdot \partial_\alpha \kappa -\dot{T}_{ (\partial_\alpha \Phi)\circ \kappa}\partial_\alpha \kappa \right) - \dot{T}_{(\partial_t \chi)\circ \kappa} \dot{\Pi}((\partial_\alpha \Phi)\circ \kappa, \partial_\alpha \kappa)\\
 &+ \dot{\Pi}((\partial_\alpha \Phi)\circ \kappa, (\partial_t \chi)\circ \kappa \cdot \partial_\alpha \kappa) +  F_{para}\\
 =& \dot{T}_{(\partial_t \chi)\circ \kappa} \partial_\alpha (\kappa^{*}\Phi)+  \dot{T}_{(\partial_t \chi)\circ \kappa}\dot{T}_{\partial_\alpha((\partial_\alpha \Phi)\circ \kappa)} \kappa - \dot{T}_{(\partial_t \chi)\circ \kappa} \dot{\Pi}((\partial_\alpha \Phi)\circ \kappa, \partial_\alpha \kappa)\\
 &+ \dot{\Pi}((\partial_\alpha \Phi)\circ \kappa, (\partial_t \chi)\circ \kappa \cdot \partial_\alpha \kappa) +  F_{para},
\end{align*}
where remainder term $F_{para}$ is given by
\begin{equation*}
 F_{para} =    \dot{T}_{(\partial_t \chi)\circ \kappa \cdot \partial_\alpha \kappa} (\partial_\alpha \Phi)\circ \kappa- \dot{T}_{(\partial_t \chi)\circ \kappa}\dot{T}_{\partial_\alpha \kappa} (\partial_\alpha \Phi)\circ \kappa.
\end{equation*} 
Using \eqref{TimeDerivativeChi} and \eqref{TimeDerivativeChiTwo},
\begin{equation*}
\|\partial_t \chi \circ \kappa \|_{L^\infty}\lesssim \| (\W, R)(t)\|_{\mathcal{H}^s}, \quad \|\partial_t\chi \circ \kappa \|_{W^{1,\infty}} \lesssim  \| (\W, R)\|_{\mathcal{H}^s}( 1+\|(\W,R)(t)\|_{\mathcal{W}^r}).
\end{equation*}
We now estimate the second to the last terms of $\dot{T}_{(\partial_\alpha \Phi)\circ \kappa}\partial_t \kappa - (\partial_\alpha \Phi \circ \kappa)\partial_t \kappa$, and show that they can be put into the source term $\tilde{F}$.
\begin{align*}
&\| \dot{T}_{(\partial_t \chi)\circ \kappa}\dot{T}_{\partial_\alpha((\partial_\alpha \Phi)\circ \kappa)} \kappa \|_{H^s}\lesssim \|(\partial_t \chi)\circ \kappa \|_{L^\infty} \|\partial_\alpha((\partial_\alpha \Phi)\circ \kappa) \|_{L^\infty} \|\partial_\alpha^2 \kappa\|_{H^{s-2}} \\
\lesssim & \|(\W, R) \|_{\mathcal{H}^s}^2 \|(\W,R)(t)\|_{\mathcal{W}^r}, \\
& \| \dot{T}_{(\partial_t \chi)\circ \kappa} \dot{\Pi}((\partial_\alpha \Phi)\circ \kappa, \partial_\alpha \kappa) \|_{H^s} \lesssim \|(\partial_t \chi)\circ \kappa \|_{L^\infty}\|(\partial_\alpha \Phi)\circ \kappa\|_{H^{s-1}}\|\partial_\alpha \kappa\|_{C^1_{*}}\\
\lesssim & \|(\W, R) \|_{\mathcal{H}^s}^2 \left(1+\|(\W,R)(t)\|_{\mathcal{W}^r} \right),\\
& \|\dot{\Pi}((\partial_\alpha \Phi)\circ \kappa, (\partial_t \chi)\circ \kappa \cdot \partial_\alpha \kappa)\|_{H^s} \lesssim \|(\partial_\alpha \Phi)\circ \kappa \|_{H^{s-1}} \|(\partial_t \chi)\circ \kappa \|_{C^1_{*}}\| \partial_\alpha \kappa \|_{C^1_{*}}\\
\lesssim &\|(\W, R) \|_{\mathcal{H}^s}^2 \left(1+\|(\W,R)(t)\|_{\mathcal{W}^r} \right),\\
&\|F_{para}\|_{H^s}\lesssim \left(\|(\partial_t \chi)\circ \kappa  \|_{L^\infty} \|\partial_\alpha \kappa \|_{W^{1,\infty}} + \|(\partial_t \chi)\circ \kappa  \|_{W^{1,\infty}} \|\partial_\alpha \kappa \|_{L^\infty}\right)\| (\partial_\alpha \Phi)\circ \kappa\|_{H^{s-1}}\\
\lesssim &\|(\W, R) \|_{\mathcal{H}^s}^2 \left(1+\|(\W,R)(t)\|_{\mathcal{W}^r} \right).
\end{align*}
Therefore, we have shown that
\begin{equation}
\kappa^{*}(\partial_t \Phi) = \left(\partial_t + T_{(\partial_t\chi)\circ \kappa}\partial_\alpha \right)\tilde{\Phi} + \tilde{F}. \label{PartialtTildePhi}
\end{equation}

Next, we consider the action of $\kappa^{*}$ on  the transport term.
According to the global para-composition Theorem \ref{t:ParaComposition}, for any $h\in \Gamma^m_\tau$, 
\begin{equation*}
\kappa^{*} T_h \Phi = T_{h^{*}}\kappa^{*}\Phi  +R_{conj}\Phi  + T_{h^{*}}R_{line}\Phi -R_{line}T_h \Phi.  
\end{equation*}
Choosing $\rho =1$, $k=2$, $r = s-\frac{1}{2}$, then
\begin{equation*}
\|\partial_\alpha \kappa -1\|_{C^1_{*}} \lesssim \|\partial_\alpha \kappa -1\|_{H^{s-\frac{1}{2}}} \lesssim \|(\W, R) \|_{\mathcal{H}^s}.
\end{equation*}
Together with the fact that $|\partial_\alpha \kappa|\geq m_0$, the two conditions in Theorem \ref{t:ParaComposition} are satisfied.

The symbol of the transport term $h = i\xi \ub(t,\alpha)$.
Choosing $\rho = \tau = \epsilon =1$, 
\begin{equation*}
h^{*}(x, \xi) = i\frac{\xi}{\partial_\alpha \kappa(t, \alpha)}\ub\circ \kappa = i(\ub\circ \kappa)(\partial_\alpha \chi \circ \kappa)\xi.
\end{equation*}
The remainder terms satisfy the estimates
\begin{align*}
&\|R_{conj}\Phi \|_{H^s} \lesssim_{m_0, \|\kappa^{'} \|_{C^\rho_{*}}} M^1_1 (h) (1+\|\partial^{2}_\alpha \kappa \|_{H^{s-\frac{3}{2}}}) \|\Phi\|_{H^s}\\
\lesssim&_{\|(\W, R) \|_{L^\infty_t \mathcal{H}^s_\alpha}} 1+\|(\W,R)\|_{\mathcal{W}^r},\\
& \|T_{h^{*}}R_{line}\Phi \|_{H^s} \lesssim M^1_0(h^{*})\|R_{line}\Phi  \|_{H^{s+1}}\\
\lesssim&_{m_0, \| \kappa^{'}\|_{C^\rho_{*}}} M^1_0(h^{*}) (1+ \|\partial^2_\alpha \kappa \|_{H^{s-\frac{3}{2}}})(\|\partial_\alpha \Phi \|_{H^{s-1}}+ \|\Phi\|_{C^r_{*}})\\
\lesssim&_{\|(\W, R) \|_{L^\infty_t \mathcal{H}^s_\alpha}} 1+\|(\W,R)\|_{\mathcal{W}^r},\\
& \|R_{line}T_h \Phi \|_{H^{s+\frac{1}{2}}}\lesssim_{m_0, \| \kappa^{'}\|_{C^\rho_{*}}} (1+ \|\partial^2_\alpha \kappa \|_{H^{s-\frac{3}{2}}})\left(\|\partial_\alpha T_h\Phi \|_{H^{s-1}}+ \|T_h\Phi \|_{C^{r-1}_{*}}\right)\\
\lesssim&_{\|(\W, R) \|_{L^\infty_t \mathcal{H}^s_\alpha}} 1+\|(\W,R)\|_{\mathcal{W}^r}.
\end{align*}
Hence, we have shown that
\begin{equation}
\kappa^{*}T_{\ub}\partial_\alpha \Phi = T_{(\ub\circ \kappa)(\partial_\alpha \chi \circ \kappa)}\partial_\alpha \tilde{\Phi}+ \tilde{F}. \label{TransportTildePhi}
\end{equation}

In the following, we estimate the conjugation of dispersive terms with $\kappa^{*}$.
One can write
\begin{equation*}
 L^{\frac{1}{2}}T_c L^{\frac{1}{2}} +\frac{\gamma}{2}\text{sgn}(D) = T_{m^{(\frac{3}{2})}} + T_{m^{(\frac{1}{2})}}, \quad m^{(\frac{3}{2})}: = c(t,\alpha)|\xi|^{\frac{3}{2}},   
\end{equation*}
and $m^{(\frac{1}{2})}$ is a symbol of order $\frac{1}{2}$.
Applying the global para-composition Theorem \ref{t:ParaComposition} to $T_{m^{(\frac{3}{2})}}$, 
\begin{equation*}
\kappa^{*} T_{m^{(\frac{3}{2})}} \Phi = T_{m^{*}}\kappa^{*}\Phi  +R_{conj}\Phi  + T_{m^{*}}R_{line}\Phi -R_{line}T_{m^{(\frac{3}{2})}} \Phi,  
\end{equation*}
where by the choice of diffeomorphism $\kappa$,
\begin{equation*}
    m^{*}(t,\alpha, \xi) = c\circ \kappa(t,\alpha) \frac{|\xi|^{\frac{3}{2}}}{\partial_\alpha \kappa} = |\xi|^{\frac{3}{2}}.
\end{equation*}
For the rest of terms in $\kappa^{*} T_{m^{(\frac{3}{2})}} \Phi$,
\begin{align*}
&\|R_{conj}\Phi \|_{H^{s-\frac{1}{2}}}  \lesssim_{m_0, \|\kappa^{'} \|_{C^\rho_{*}}} M^{\frac{3}{2}}_1 (m^{(\frac{3}{2})}) \left(1+\|\partial^{2}_\alpha \kappa \|_{H^{s-\frac{3}{2}}}\right) \|\Phi\|_{H^s} \lesssim \|(\W, R) \|_{\mathcal{H}^s},\\
&\|T_{m^{*}}R_{line}\Phi \|_{H^{s-\frac{1}{2}}} \lesssim_{m_0, \| \kappa^{'}\|_{C^\rho_{*}}} \left(1+ \|\partial^2_\alpha \kappa \|_{H^{s-\frac{3}{2}}}\right)(\|\partial_\alpha \Phi \|_{H^{s-1}}+ \|\Phi\|_{C^r_{*}})\\
\lesssim&_{\|(\W, R) \|_{L^\infty_t \mathcal{H}^s_\alpha}} 1+\|(\W,R)\|_{\mathcal{W}^r},\\
& \|R_{line}T_{m^{(\frac{3}{2})}} \Phi \|_{H^{s-\frac{1}{2}}} \lesssim_{m_0, \| \kappa^{'}\|_{C^\rho_{*}}} \left(1+ \|\partial^2_\alpha \kappa \|_{H^{s-\frac{3}{2}}}\right)\left(\|T_{m^{(\frac{3}{2})}}\Phi \|_{H^{s-\frac{3}{2}}}+ \|T_{m^{(\frac{3}{2})}}\Phi\|_{C^{\frac{1}{2}}_{*}}\right)\\
\lesssim&_{\|(\W, R) \|_{L^\infty_t \mathcal{H}^s_\alpha}} 1+ \|(\W,R)\|_{\mathcal{W}^r},
\end{align*}
so that they can be put into $\tilde{F}$.
Since $m^{(\frac{1}{2})}\in \Gamma^{\frac{1}{2}}_0$, by the operation and the linearization properties in Theorem \ref{t:ParaComposition}, 
\begin{equation*}
\|i\kappa^{*}T_{m^{(\frac{1}{2})}}\Phi \|_{H^{s-\frac{1}{2}}} \lesssim \|(\W, R) \|_{H^s}.
\end{equation*}
Therefore, we obtain
\begin{equation}
i\kappa^{*}\left(L^{\frac{1}{2}}T_c L^{\frac{1}{2}} +\frac{\gamma}{2}\text{sgn}(D) \right)\Phi = i|D|^{\frac{3}{2}}\tilde{\Phi} + \tilde{F}. \label{GammaTildePhi}
\end{equation}

Finally, for the source term on the right-hand side of \eqref{e:SinglePara}, using \eqref{FHsBound},
\begin{equation*}
 \|\kappa^{*}F (t)\|_{H^{s}} \lesssim_{\|(\W, R) \|_{L^\infty_t \mathcal{H}^s_\alpha}} 1+\|(\W,R)(t)\|_{\mathcal{W}^r},
\end{equation*}
and $\kappa^{*}F$ can be put into $\tilde{F}$.

Gathering the equations \eqref{PartialtTildePhi}, \eqref{TransportTildePhi} and \eqref{GammaTildePhi}, we finish the proof of this proposition.
\end{proof}

In order to obtain the Strichartz estimate for the nonlinear paradifferential equation of \eqref{ConstantDispersive} type, we need Theorem 4.14 from \cite{MR3724757}, which is recalled below. 
\begin{proposition}[The Strichartz estimate \cite{MR3724757}] \label{t:StrichartzGeneral}
 Let $I=[0,T]$ and $s_0\in \mathbb{R}$, $V\in L^\infty([0,T];L^\infty(\mathbb{R}))\cap L^4([0,T]; W^{1,\infty}(\mathbb{R}))$ and $f\in L^4(I; H^{s_0 -\frac{1}{2}}(\mathbb{R}))$.
 If $u\in L^\infty(I;H^{s_0}(\mathbb{R}))$  is a solution to the problem
 \begin{equation*}
 \left(\partial_t +T_V \partial_x -i|D_x|^{\frac{3}{2}} \right)u = f,
 \end{equation*}
 then for every $\epsilon>0$, there exists a constant $C$ that depends on $\| V\|_{L^\infty([0,T];L^\infty)} + \| V\|_{ L^4([0,T]; W^{1,\infty})}$, such that
 \begin{equation*}
 \| u\|_{L^4(I; C_{*}^{s_0 -\frac{1}{4}-\epsilon})} \leq C \left( \|f \|_{L^4(I; H^{s_0 -\frac{1}{2}-\epsilon})} + \| u\|_{L^\infty(I; H^{s_0})} \right).
 \end{equation*}
\end{proposition}

Now we are in the position to finish the proof of Theorem \ref{t:MainTwo}.
\begin{proof}[Proof of Therorem \ref{t:MainTwo}]
Applying Proposition \ref{t:StrichartzGeneral} for \eqref{ConstantDispersive}, we obtain the estimate
 \begin{equation}
 \| \tilde{\Phi}\|_{L^4(I; C_{*}^{s -\frac{1}{4}-\epsilon})} \leq C \left( \|\tilde{F} \|_{L^4(I; H^{s -\frac{1}{2}-\epsilon})} + \| \tilde{\Phi}\|_{L^\infty(I; H^{s})} \right), \quad s>2, \label{TildePhiStrichartz}
 \end{equation}
 where the constant $C$ depends on
\begin{equation*}
\|(\ub \circ \kappa)(\partial_\alpha \chi \circ \kappa) + \partial_t \chi \circ \kappa \|_{L^\infty([0,T];L^\infty)} + \| (\ub \circ \kappa)(\partial_\alpha \chi \circ \kappa) + \partial_t \chi \circ \kappa\|_{ L^4([0,T]; W^{1,\infty})}.
 \end{equation*}

For the right-hand side of \eqref{TildePhiStrichartz}, using Theorem \ref{t:ParaComposition} and \eqref{varphiLTwo},
\begin{align*}
&\|\tilde{F} \|_{L^4(I; H^{s -\frac{1}{2}-\epsilon})} \lesssim_{\|(\W, R) \|_{L^\infty_t \mathcal{H}^s_\alpha}} 1+\|(\W, R)\|_{L^4(I;\mathcal{W}^r)}, \\ 
& \| \tilde{\Phi}\|_{L^\infty(I; H^{s})} \leq \|\kappa_g^{*}\Phi \|_{L^\infty(I; H^{s})} + \|R_{line}\Phi \|_{L^\infty(I; H^{s})} \lesssim \| (\W,R)\|_{L_t^\infty(I; \mathcal{H}^s)}, \\
&\|(\ub \circ \kappa)(\partial_\alpha \chi \circ \kappa) + \partial_t \chi \circ \kappa \|_{L^\infty([0,T];L^\infty)} \leq \|\ub  \|_{L^\infty_{t,\alpha}} \|\partial_\alpha \chi  \|_{L^\infty_{t,\alpha}} + \| \partial_t \chi \|_{L^\infty_{t,\alpha}}\\
\lesssim& \|(\W,R) \|_{L_t^\infty(I; \mathcal{H}^s)}^2 + \|(\W,R) \|_{L_t^\infty(I; \mathcal{H}^s)}.
\end{align*}
To estimate $ (\ub \circ \kappa)(\partial_\alpha \chi \circ \kappa) + \partial_t \chi \circ \kappa$ in $ L^4(I; W^{1,\infty})$, note that
\begin{equation*}
\partial_\alpha((\ub \circ \kappa)(\partial_\alpha \chi \circ \kappa) + \partial_t \chi \circ \kappa) =  (\partial_\alpha \ub \circ \kappa)(\partial_\alpha \chi \circ \kappa)\partial_\alpha \kappa +  (\ub \circ \kappa)(\partial_\alpha^2 \chi \circ \kappa)\partial_\alpha \kappa  + \partial_t\partial_\alpha \chi \circ \kappa \cdot \partial_\alpha \kappa,
\end{equation*}
so that we can bound using \eqref{UbCOneStar}, \eqref{UbHsEst}, \eqref{HsChiAlpha}, and \eqref{TimeDerivativeChiTwo},
\begin{align*}
&\|(\ub \circ \kappa)(\partial_\alpha \chi \circ \kappa) + \partial_t \chi \circ \kappa \|_{L^4(I;L^\infty)} \leq T^{\frac{1}{4}} |(\ub \circ \kappa)(\partial_\alpha \chi \circ \kappa) + \partial_t \chi \circ \kappa \|_{L^\infty(I;L^\infty)} \\
\lesssim&  T^{\frac{1}{4}}\|(\W,R) \|_{L_t^\infty(I; \mathcal{H}^s)}^2 +  T^{\frac{1}{4}}\|(\W,R) \|_{L_t^\infty(I; \mathcal{H}^s)},\\
& \|(\partial_\alpha \ub \circ \kappa)(\partial_\alpha \chi \circ \kappa)\partial_\alpha \kappa \|_{L^4(I;L^\infty)} \leq \|\partial_\alpha \ub \|_{L^4(I;L^\infty)} \|\partial_\alpha \chi \|_{L^\infty_{t,\alpha}} \|\partial_\alpha \kappa\|_{L^\infty_{t,\alpha}} \\
\lesssim& \|(\W, R)\|_{L^4(I;\mathcal{W}^r)}\left(1+\|(\W,R)\|_{L^\infty(I;\mathcal{H}^s)}\right)^2,\\
& \| (\ub \circ \kappa)(\partial_\alpha^2 \chi \circ \kappa)\partial_\alpha \kappa\|_{L^4(I;L^\infty)}\leq \|\ub \|_{L^\infty_{t,\alpha}} \|\partial_\alpha^2 \chi \|_{L^4(I;L^\infty)} \|\partial_\alpha \kappa\|_{L^\infty_{t,\alpha}}\\
\lesssim& \|(\W, R)\|_{L^4(I;\mathcal{W}^r)}\left(1+\|(\W,R)\|_{L^\infty(I;\mathcal{H}^s)}\right)^2, \\
&\|\partial_t\partial_\alpha \chi \circ \kappa \cdot \partial_\alpha \kappa \|_{L^4(I;L^\infty)}\leq \|\partial_t \partial_\alpha \chi \|_{L^4(I; L^\infty)} \|\partial_\alpha \kappa \|_{L^\infty(I;L^\infty)}\\
\lesssim& \|(\W, R)\|_{L^4(I;\mathcal{W}^r)}\left(1+\|(\W,R)\|_{L^\infty(I;\mathcal{H}^s)}\right).
\end{align*}
Therefore, we have shown that
\begin{equation*}
  \| \tilde{\Phi}\|_{L^4(I; C_{*}^{s -\frac{1}{4}-\epsilon})} \leq \tilde{C},   
\end{equation*}
for some constant $\tilde{C}$ that depends on
\begin{equation*}
   \| (\W, R)\|_{ C^0(I; \mathcal{H}^s)} + \| (\W, R)\|_{ L^4(I; \mathcal{W}^r)}. 
\end{equation*}
It only suffices to obtain a lower bound for the left-hand side of \eqref{TildePhiStrichartz}.
By the definition of $\tilde{\Phi}$, 
\begin{equation*}
    \Phi \circ \kappa = \tilde{\Phi}+ \dot{T}_{\partial_\alpha \Phi\circ \kappa}\kappa.
\end{equation*}
Since we can estimate using the Sobolev inequality
\begin{equation*}
\|\dot{T}_{\partial_\alpha \Phi\circ \kappa}\kappa \|_{L^4(I; C_{*}^{s -\frac{1}{4}-\epsilon})} \lesssim \|\Phi\circ\kappa \|_{L^\infty_{t,\alpha}} \|\partial_\alpha \kappa\|_{L^4(I; C_{*}^{s -\frac{1}{4}-\epsilon} )}\lesssim T^{\frac{1}{4}}\|(\W,R) \|_{L^\infty(I; \mathcal{H}^s)}\|\W \|_{L^\infty(I; H^{s+\frac{1}{2}})},
\end{equation*}
we get $\Phi\circ \kappa \lesssim \tilde{C}$, so that by lemma \ref{t:Composition},
\begin{equation*}
\|T_p \eta \|_{L^4(I; C_{*}^{s -\frac{1}{4}-\epsilon} )}+ \|T_q \zeta \|_{L^4(I; C_{*}^{s -\frac{1}{4}-\epsilon} )} \lesssim\|\Phi \|_{L^4(I; C_{*}^{s -\frac{1}{4}-\epsilon} )} = \| (\Phi\circ \kappa) \circ \chi\|_{L^4(I; C_{*}^{s -\frac{1}{4}-\epsilon} )} \lesssim \tilde{C}.
\end{equation*}
To obtain an estimate for $\eta$, we only need to consider the leading part of the analysis for $T_{p^{(\frac{1}{2})}} \eta$,
\begin{equation*}
    \|T_{p^{(\frac{1}{2})}} \eta \|_{L^4(I; C_{*}^{s -\frac{1}{4}-\epsilon} )} \lesssim  \tilde{C}+  \|T_{p^{(-\frac{1}{2})}} \eta \|_{L^4(I; C_{*}^{s -\frac{1}{4}-\epsilon} )} \lesssim \tilde{C}.
\end{equation*}
$p^{(\frac{1}{2})}$ is an elliptic symbol,  $T_{1/ p^{(\frac{1}{2})}}$ is well-defined and is of order $-\frac{1}{2}$.
Thus,
\begin{equation*}
    \|\eta\|_{L^4(I; C_{*}^{s +\frac{1}{4}-\epsilon} )} \lesssim \|T_{1/ p^{(\frac{1}{2})}} T_{p^{(\frac{1}{2})}} \eta\|_{L^4(I; C_{*}^{s +\frac{1}{4}-\epsilon} )} + \|S \eta\|_{L^4(I; C_{*}^{s -\frac{1}{4}-\epsilon} )},
\end{equation*}
where $S  = \left(I - T_{1/ p^{(\frac{1}{2})}} T_{p^{(\frac{1}{2})}} \right)$ is an operator of order $-1$.
As a consequence, 
\begin{equation*}
 \|\eta\|_{L^4(I; C_{*}^{s +\frac{1}{4}-\epsilon} )} \lesssim \| T_{p^{(\frac{1}{2})}} \eta\|_{L^4(I; C_{*}^{s -\frac{1}{4}-\epsilon} )} + \|S \eta\|_{L^4(I; C_{*}^{s -\frac{1}{4}-\epsilon} )} \lesssim \tilde{C}.
\end{equation*}
Similarly, one can get
\begin{equation*}
 \|\zeta\|_{L^4(I; C_{*}^{s -\frac{1}{4}-\epsilon} )} \lesssim \tilde{C}.    
\end{equation*}
Finally,
\begin{equation*}
    \|(\W, R) \|_{L^4(I; C_{*}^{s +\frac{1}{4}-\epsilon} \times  C_{*}^{s -\frac{1}{4}-\epsilon} )} \lesssim  \|(\eta, \zeta) \|_{L^4(I; C_{*}^{s +\frac{1}{4}-\epsilon} \times  C_{*}^{s -\frac{1}{4}-\epsilon} )} \leq \tilde{C}^{'},
\end{equation*}
for some constant $\tilde{C}^{'}$ that depends on
\begin{equation*}
   \| (\W, R)\|_{ C^0(I; \mathcal{H}^s)} + \| (\W, R)\|_{ L^4(I; \mathcal{W}^r)}. 
\end{equation*}
This finishes the proof of Theorem \ref{t:MainTwo}, due to the fact that the space $W^{s,\infty}(\mathbb{R})$ coincides with $C^s_{*}(\mathbb{R})$ for non-integer $s$.
 \end{proof}

 \section{The local well-posedness of the Cauchy problem} \label{s:Cauchy}
 In this section, we follows the steps in the expository article of Ifrim-Tataru \cite{MR4557379}, and sketch the proof of Theorem \ref{t:MainWellPosed}.
For the details of the proof and the use of \textit{frequency envelopes}, we refer the interested readers to Section 5.1 in \cite{MR4557379}.
 The proof of the local well-posedness is divided into several steps.

 First, we truncate the frequency of initial data at $2^N$ and obtain a sequence of approximate solutions that converge weakly to the regular solution.
 Second, we obtain a contraction estimate for two different solutions of \eqref{e:WR} and therefore get the uniqueness in $\mathcal{H}^s$ for $s>\frac{3}{2}$.
 Next, we use the regular solutions to construct the rough solutions in $\mathcal{H}^s(\mathbb{R})$ for $\frac{3}{2}\geq s>\frac{5}{4}$.
 Finally, we prove the continuous dependence on the initial data.

 \subsection{Existence of regular solutions}
 Consider the initial data $(\W_0, R_0) \in \mathcal{H}^{\tau}$.
 For each positive integer $N$, we choose the frequency projection $P_{<N}$ that selects the frequencies less or similar than $2^N$.
 Then for fixed $N$, the  mollified system
  \begin{equation*} 
\left\{
\begin{aligned}
 &  T_{D_t}^{(<N)}\W +\partial_\alpha P_{<N}T_{(1+\W)(1-\bar{Y})}R = G_{<N}\\
& T_{D_t}^{(<N)} \W  +i\gamma P_{<N}R +i\sigma P_{<N}T_{J^{-\frac{1}{2}}(1-Y)^2}\W_{\alpha \alpha}- 3i\sigma P_{<N}T_{J^{-\frac{1}{2}}(1-Y)^3 \W_\alpha} \W_\alpha - igP_{<N}\W = K_{<N}, 
\end{aligned}
\right.
\end{equation*} 
where 
\begin{equation*}
T_{D_t}^{(<N)} = \partial_t + T_{\ub_N} \partial_\alpha, \quad \ub_N = P_{<N}\ub(P_{<N}\W, P_{<N}R),    
\end{equation*}
and 
\begin{equation*}
G_{<N} = P_{<N} G(P_{<N}\W, P_{<N}R), \quad   K_{<N} = P_{<N} K(P_{<N}\W, P_{<N}R),   
\end{equation*}
form a system of ordinary differential equations in $\mathcal{H}^s$.
The mollified system admits a local solution $(\W^{N}, R^{N})$.
Using the argument as in Theorem \ref{t:EnergyEstimate} and Theorem \ref{t:MainTwo}, one can show that for all $N\in \mathbb{N}_{+}$,
\begin{equation}
    \|(\W^{N}, R^{N} ) \|_{L^\infty(I;\mathcal{H}^s)} + \|(\W^{N}, R^{N} ) \|_{L^4(I;\mathcal{W}^r)}< +\infty, \quad 1<r<s-\frac{1}{4} \label{UniformBound}
\end{equation}
is uniformly bounded from above. 
Then the regular solution $(\W, R)$ is a weak limit on a subsequence of $(\W^{N}, R^{N} )$ as $N\rightarrow \infty$.

\subsection{Contraction estimate of the solution}
 Having constructed the regular solution of \eqref{e:WR}, we now sketch the proof of the contraction estimate for two solutions in weaker norms.
The uniqueness  of solutions of \eqref{e:WR} in $\mathcal{H}^s$ with $s>\frac{3}{2}$ follows directly from the contraction estimates.

 \begin{proposition} \label{t:Contraction}
 Let $(\W_j, R_j)$, $j=1,2$ be two solutions of \eqref{e:WR} on $I =[0,T]$, such that for $r>1, s>\frac{3}{2}$, 
 \begin{equation*}
(\W_j, R_j) \in L^\infty(I; \mathcal{H}^s(\mathbb{R}))\cap L^4(I; \mathcal{W}^r(\mathbb{R})).
\end{equation*}
Set Sobolev and H\"{o}lder control norms
\begin{equation*}
M^j_{s,T}: = \|(\W_j, R_j) \|_{L^\infty(I; \mathcal{H}^s)}, \quad N^j_{r,T}: = \|(\W_j, R_j) \|_{L^4(I;  \mathcal{W}^r)}.
\end{equation*}
Let $(\delta \W, \delta R): = (\W_1 -\W_2, R_1 -R_2)$.
Then for some constant $C$ that depends on $M^j_{s,T}, N^j_{r,T}$, $j=1,2$,
\begin{equation*}
   \|(\delta \W, \delta R) \|_{L^\infty(I; \mathcal{H}^{s-\frac{3}{2}})} + \|(\delta \W, \delta R) \|_{L^4(I; \mathcal{W}^{r-\frac{3}{2}})}  \leq C\| (\delta \W_0, \delta R_0)\|_{\mathcal{H}^{s-\frac{3}{2}}}.
\end{equation*}
 \end{proposition}
In the following, for simplicity, when we write $f_j$, $j=1,2$, this means $f_j$ is a function of $(\W_j, R_j)$.
We also set $\delta f := f_1 -f_2$.

By direct computation using the definition, for $\mu>0$, we have the contraction estimates for the auxiliary functions:
\begin{align*}
&\|\delta Y\|_{H^\mu}\lesssim \|\delta \W \|_{H^\mu}, \quad \|\delta Y \|_{C^\mu_{*}} \lesssim \|\delta \W \|_{C^\mu_{*}}, \\
&\|\delta \ua\|_{H^\mu}\lesssim (\CalB_1 + \CalB_2)\|(\delta \W, \delta R) \|_{\mathcal{H}^\mu}, \quad \|\delta \ua\|_{C^\mu_{*}}\lesssim (\CalB_1 + \CalB_2)\|(\delta \W, \delta R) \|_{C^{\mu+\frac{1}{2}}_{*}\times C^\mu_{*}},\\
&\|\delta \ub\|_{H^\mu}\lesssim_{\CalA_1, \CalA_2}\|(\delta \W, \delta R) \|_{\mathcal{H}^\mu}, \quad \|\delta \ub\|_{C^\mu_{*}}\lesssim_{\CalA_1, \CalA_2}\|(\delta \W, \delta R) \|_{C^{\mu+\frac{1}{2}}_{*}\times C^\mu_{*}}.
\end{align*}

Recall that in Section \ref{s:reduction}, we paralinearize the water waves and rewrite it as the paradifferential system \eqref{e:ParaWater}:
 \begin{equation*} 
\left\{
\begin{aligned}
 &  T_{D_t} \W +\partial_\alpha T_{(1+\W)(1-\bar{Y})}R = G\\
& T_{D_t} R  +i\gamma R +i\sigma T_{J^{-\frac{1}{2}}(1-Y)^2}\W_{\alpha \alpha}- 3i\sigma T_{J^{-\frac{1}{2}}(1-Y)^3 \W_\alpha} \W_\alpha - ig\W =  K, 
\end{aligned}
\right.
\end{equation*}  
where the source terms $(G,K)$ are given by
\begin{align*}
&G =  \partial_\alpha (T_{(1+\W)(1-\bar{Y})}-T_{1+\W}T_{1-\bar{Y}})R -T_{\ub_\alpha}W  +\partial_\alpha T_{1+\W}\nP\Pi(\bar{Y}, R) - \partial_\alpha\nP\Pi(\W, \ub) \\
  +&i\frac{\gamma}{2}\partial_\alpha T_{\W}W - i\frac{\gamma}{2}\partial_\alpha T_{1+\W}T_{\bar{Y}}W - i\frac{\gamma}{2}\partial_\alpha T_{1+\W}\nP\Pi(W,\bar{Y}),\\
&K = -\nP T_{R_\alpha}\ub -\nP\Pi(R_\alpha, \ub)+i\nP[\ua Y]-igY\W -\nP[R\bar{R}_\alpha] -i\frac{\gamma}{2}\nP[W\bar{R}_\alpha-\bar{W}_\alpha R]\\
-& i\sigma (1-Y)\nP[ J^{-\frac{1}{2}} \W_{ \alpha}(1-Y)]_{\alpha} +i\sigma (1-Y) \nP [J^{-\frac{1}{2}}\bar{\W}_{ \alpha}(1-\bar{Y})]_{\alpha} +i\sigma T_{J^{-\frac{1}{2}}(1-Y)^2}\W_{\alpha \alpha}\\
-& 3i\sigma T_{J^{-\frac{1}{2}}(1-Y)^3 \W_\alpha} \W_\alpha.
\end{align*}

The difference of two solutions $(\delta \W, \delta R)$ satisfies the system
 \begin{equation*} 
\left\{
\begin{aligned}
 &  (\partial_t + T_{\ub_{1}}\partial_\alpha) \delta\W +\partial_\alpha T_{(1+\W_1)(1-\bar{Y}_1)}\delta R = G^\sharp\\
& (\partial_t + T_{\ub_{1}}\partial_\alpha) \delta R  +i\gamma \delta R +i\sigma T_{J_{1}^{-\frac{1}{2}}(1-Y_1)^2}\delta \W_{\alpha \alpha}- 3i\sigma T_{J_{1}^{-\frac{1}{2}}(1-Y_1)^3 \W_{1,\alpha}} \delta\W_\alpha - ig\delta \W =  K^\sharp, 
\end{aligned}
\right.
\end{equation*}  
where the source terms $(G^\sharp, K^\sharp)$ of  difference equations are
\begin{align*}
G^\sharp &= \delta G -T_{\delta \ub}\partial_\alpha\W_2 + \partial_\alpha T_{(1+\W_2)(1-\bar{Y}_2)-(1+\W_1)(1-\bar{Y}_1)}R_2, \\
K^\sharp &= \delta K -T_{\delta \ub}\partial_\alpha R_2 +i\sigma T_{J_{2}^{-\frac{1}{2}}(1-Y_2)^2 -J_{1}^{-\frac{1}{2}}(1-Y_1)^2} \W_{2, \alpha \alpha} +3i\sigma T_{J_{1}^{-\frac{1}{2}}(1-Y_1)^3 \W_{1,\alpha}-J_{2}^{-\frac{1}{2}}(1-Y_2)^3 \W_{2,\alpha}} \W_{2,\alpha}.
\end{align*}

As in the computation in Section \ref{s:reduction}, one can obtain the difference bound
\begin{align*}
\|( G^\sharp, K^\sharp)(t)\|_{\mathcal{H}^{s-\frac{3}{2}}} &\lesssim  \left(\| (\W_1 ,  R_1)(t) \|_{\mathcal{W}^{r}} + \| (\W_2 ,  R_2)(t) \|_{\mathcal{W}^{r}}\right)\|(\delta \W, \delta R)(t) \|_{\mathcal{H}^{s-\frac{3}{2}}} \\
&+ \left(\| (\W_1 ,  R_1)(t) \|_{\mathcal{H}^{s}} + \| (\W_2 ,  R_2)(t) \|_{\mathcal{H}^{s}}\right)\|(\delta \W, \delta R)(t) \|_{\mathcal{W}^{r-\frac{3}{2}}}. 
\end{align*}

Rewriting the difference equation in Wahl\'{e}n variables and using the symmetrization as in Section \ref{s:reduction}, we can then obtain the a priori energy estimate and the Strichartz estimate.
The energy estimate gives
\begin{equation}
\begin{aligned}
 \|(\delta \W, \delta R)(T) \|_{\mathcal{H}^{s-\frac{3}{2}}} &\lesssim \Big(\| (\delta \W_0, \delta R_0)\|_{\mathcal{H}^{s-\frac{3}{2}}} \\
 +& T^{\frac{3}{4}}(\| (\delta \W, \delta R)\|_{L^4(I;\mathcal{W}^{r-\frac{3}{2}})}+ \| (\delta \W, \delta R)\|_{L^\infty(I;\mathcal{H}^{s-\frac{3}{2}})})\Big),   
\end{aligned} \label{EnergyDiff}
\end{equation}
where the implicit constant depends on $M^j_{s,T}, N^j_{r,T}$, $j=1,2$.
Furthermore, by the Strichartz estimate, 
\begin{equation}
\| (\delta \W, \delta R)\|_{L^4(I;\mathcal{W}^{r-\frac{3}{2}})} \lesssim T^{\delta} \left(\| (\delta \W, \delta R)\|_{L^4(I;\mathcal{W}^{r-\frac{3}{2}})}+ \| (\delta \W, \delta R)\|_{L^\infty(I;\mathcal{H}^{s-\frac{3}{2}})} \right), \label{StrichartzDiff}
\end{equation}
for some time $0<T<1$ with $\delta>0$.
Separating the time interval $I = [0,T]$ into $[0,T_1], [T_1, T_2], \cdots [T_n, T]$, on each interval we combine the energy estimate \eqref{EnergyDiff} and the Strichartz estimate \eqref{StrichartzDiff}, which gives the proof of Proposition \ref{t:Contraction}.  

\subsection{Construction of rough solutions}
Following the work in \cite{MR3535894} as well as \cite{MR4557379}, we truncate the frequency the initial data at $2^{N}$ for $(W_{<N}(0), Q_{<N}(0))$ and $(\W_{<N}(0), R_{<N}(0))$.
We can then establish the bound of regularized initial data using frequency envelopes.
The corresponding solutions will be regular, with a uniform lifespan bound as in \eqref{UniformBound}.
Using the energy bounds for the solutions and the difference bounds $(\W^{N+1}-\W^{N}, R^{N+1}-R^{N})$ in $\mathcal{H}^{s}$, the sequence $(\W^{N}, R^{N})$ converges to a solution $(\W, R)$ with uniform $\mathcal{H}^{s}$ bound in time interval $[0, T]$.
This further shows that  when $\frac{5}{4}<s \leq \frac{3}{2}$, the solution is unique in $\mathcal{H}^s$ in the sense that it is the unique limit of regular solutions.

\subsection{Continuous dependence on the data} 
For any sequences $(\W_{j,0}, R_{j,0})$ that converge to initial data $(\W_0, R_0)$ in $\mathcal{H}^{s}$, we truncate the data at frequency $2^N$ and obtain the corresponding regular solutions $(\W_j^N, R_j^N)$, respectively $(\W^N, R^N)$.
we have the continuous dependence on the data for the regular solutions, for each $N$,
\begin{equation*}
(\W^N_j, R^N_j)- (\W^N, R^N) \rightarrow 0 \quad \text{in } \mathcal{H}^\tau, \quad \tau> \frac{3}{2}.
\end{equation*}
On the other hand, for the initial data, we let $N$ go to infinity,
\begin{equation*}
(\W^N_{j,0}, R^N_{j,0})- (\W_{j,0}, R_{j,0})
\rightarrow 0 \quad \text{in } \mathcal{H}^s, \quad \text{uniformly in } j.
\end{equation*}
Using the notation of frequency envelopes, we get the uniform convergence for the solution:
\begin{equation*}
(\W^N_j, R^N_j)- (\W_j, R_j) \rightarrow 0 \quad \text{in } \mathcal{H}^s, \quad \text{uniformly in } j.
\end{equation*}
 We can again let $N$ go to infinity to  further conclude that
 \begin{equation*}
  (\W_j, R_j)-(\W, R)\rightarrow 0 \quad \text{in } \mathcal{H}^{s}.
 \end{equation*}
 This shows the continuous dependence on the data for solutions in $\mathcal{H}^{s}$.

\bibliography{ww}

\begin{thebibliography}{10}

\bibitem{MR4244258}
Siddhant Agrawal.
\newblock Angled crested like water waves with surface tension: wellposedness of the problem.
\newblock {\em Comm. Math. Phys.}, 383(3):1409--1526, 2021.

\bibitem{MR4684336}
Siddhant Agrawal.
\newblock Angled crested like water waves with surface tension {II}: {Z}ero surface tension limit.
\newblock {\em Mem. Amer. Math. Soc.}, 293(1458):v+124, 2024.

\bibitem{ai2023improved}
Albert Ai.
\newblock Improved low regularity theory for gravity-capillary waves, 2023.

\bibitem{MR4483135}
Albert Ai, Mihaela Ifrim, and Daniel Tataru.
\newblock Two-dimensional gravity waves at low regularity {II}: {G}lobal solutions.
\newblock {\em Ann. Inst. H. Poincar\'{e} C Anal. Non Lin\'{e}aire}, 39(4):819--884, 2022.

\bibitem{ai2023dimensional}
Albert Ai, Mihaela Ifrim, and Daniel Tataru.
\newblock Two dimensional gravity waves at low regularity {I}: Energy estimates, 2023.

\bibitem{MR2805065}
T.~Alazard, N.~Burq, and C.~Zuily.
\newblock On the water-wave equations with surface tension.
\newblock {\em Duke Math. J.}, 158(3):413--499, 2011.

\bibitem{MR3260858}
T.~Alazard, N.~Burq, and C.~Zuily.
\newblock On the {C}auchy problem for gravity water waves.
\newblock {\em Invent. Math.}, 198(1):71--163, 2014.

\bibitem{MR3776276}
Thomas Alazard, Pietro Baldi, and Daniel Han-Kwan.
\newblock Control of water waves.
\newblock {\em J. Eur. Math. Soc. (JEMS)}, 20(3):657--745, 2018.

\bibitem{MR2931520}
Thomas Alazard, Nicolas Burq, and Claude Zuily.
\newblock Strichartz estimates for water waves.
\newblock {\em Ann. Sci. \'{E}c. Norm. Sup\'{e}r. (4)}, 44(5):855--903, 2011.

\bibitem{MR0814548}
S.~Alinhac.
\newblock Paracomposition et op\'{e}rateurs paradiff\'{e}rentiels.
\newblock {\em Comm. Partial Differential Equations}, 11(1):87--121, 1986.

\bibitem{MR2162781}
David~M. Ambrose and Nader Masmoudi.
\newblock The zero surface tension limit of two-dimensional water waves.
\newblock {\em Comm. Pure Appl. Math.}, 58(10):1287--1315, 2005.

\bibitem{MR2768550}
Hajer Bahouri, Jean-Yves Chemin, and Rapha\"{e}l Danchin.
\newblock {\em Fourier analysis and nonlinear partial differential equations}, volume 343 of {\em Grundlehren der mathematischen Wissenschaften [Fundamental Principles of Mathematical Sciences]}.
\newblock Springer, Heidelberg, 2011.

\bibitem{MR4228858}
M.~Berti, L.~Franzoi, and A.~Maspero.
\newblock Traveling quasi-periodic water waves with constant vorticity.
\newblock {\em Arch. Ration. Mech. Anal.}, 240(1):99--202, 2021.

\bibitem{MR4658635}
Massimiliano Berti, Alberto Maspero, and Federico Murgante.
\newblock Hamiltonian paradifferential {B}irkhoff normal form for water waves.
\newblock {\em Regul. Chaotic Dyn.}, 28(4-5):543--560, 2023.

\bibitem{MR1637554}
Klaus Beyer and Matthias G\"{u}nther.
\newblock On the {C}auchy problem for a capillary drop. {I}. {I}rrotational motion.
\newblock {\em Math. Methods Appl. Sci.}, 21(12):1149--1183, 1998.

\bibitem{bieri2015motion}
Lydia Bieri, Shuang Miao, Sohrab Shahshahani, and Sijue Wu.
\newblock On the motion of a self-gravitating incompressible fluid with free boundary and constant vorticity: An appendix, 2015.

\bibitem{MR2291920}
Daniel Coutand and Steve Shkoller.
\newblock Well-posedness of the free-surface incompressible {E}uler equations with or without surface tension.
\newblock {\em J. Amer. Math. Soc.}, 20(3):829--930, 2007.

\bibitem{MR1158383}
W.~Craig, C.~Sulem, and P.-L. Sulem.
\newblock Nonlinear modulation of gravity waves: a rigorous approach.
\newblock {\em Nonlinearity}, 5(2):497--522, 1992.

\bibitem{MR3487264}
Thibault de~Poyferr\'{e} and Quang-Huy Nguyen.
\newblock Strichartz estimates and local existence for the gravity-capillary waves with non-{L}ipschitz initial velocity.
\newblock {\em J. Differential Equations}, 261(1):396--438, 2016.

\bibitem{MR3770970}
Thibault de~Poyferr\'{e} and Quang-Huy Nguyen.
\newblock A paradifferential reduction for the gravity-capillary waves system at low regularity and applications.
\newblock {\em Bull. Soc. Math. France}, 145(4):643--710, 2017.

\bibitem{MR3415532}
M.~D. Groves and E.~Wahl\'{e}n.
\newblock Existence and conditional energetic stability of solitary gravity-capillary water waves with constant vorticity.
\newblock {\em Proc. Roy. Soc. Edinburgh Sect. A}, 145(4):791--883, 2015.

\bibitem{MR3625189}
Benjamin Harrop-Griffiths, Mihaela Ifrim, and Daniel Tataru.
\newblock Finite depth gravity water waves in holomorphic coordinates.
\newblock {\em Ann. PDE}, 3(1):Paper No. 4, 102, 2017.

\bibitem{MR3535894}
John~K. Hunter, Mihaela Ifrim, and Daniel Tataru.
\newblock Two dimensional water waves in holomorphic coordinates.
\newblock {\em Comm. Math. Phys.}, 346(2):483--552, 2016.

\bibitem{MR4462478}
Mihaela Ifrim, James Rowan, Daniel Tataru, and Lizhe Wan.
\newblock The {B}enjamin-{O}no approximation for 2{D} gravity water waves with constant vorticity.
\newblock {\em Ars Inven. Anal.}, pages Paper No. 3, 33, 2022.

\bibitem{MR3499085}
Mihaela Ifrim and Daniel Tataru.
\newblock Two dimensional water waves in holomorphic coordinates {II}: {G}lobal solutions.
\newblock {\em Bull. Soc. Math. France}, 144(2):369--394, 2016.

\bibitem{MR3667289}
Mihaela Ifrim and Daniel Tataru.
\newblock The lifespan of small data solutions in two dimensional capillary water waves.
\newblock {\em Arch. Ration. Mech. Anal.}, 225(3):1279--1346, 2017.

\bibitem{MR3869381}
Mihaela Ifrim and Daniel Tataru.
\newblock Two-dimensional gravity water waves with constant vorticity {I}: {C}ubic lifespan.
\newblock {\em Anal. PDE}, 12(4):903--967, 2019.

\bibitem{MR4557379}
Mihaela Ifrim and Daniel Tataru.
\newblock Local well-posedness for quasi-linear problems: a primer.
\newblock {\em Bull. Amer. Math. Soc. (N.S.)}, 60(2):167--194, 2023.

\bibitem{MR3060183}
David Lannes.
\newblock {\em The water waves problem}, volume 188 of {\em Mathematical Surveys and Monographs}.
\newblock American Mathematical Society, Providence, RI, 2013.
\newblock Mathematical analysis and asymptotics.

\bibitem{MR2969824}
Calin~Iulian Martin.
\newblock Local bifurcation and regularity for steady periodic capillary-gravity water waves with constant vorticity.
\newblock {\em Nonlinear Anal. Real World Appl.}, 14(1):131--149, 2013.

\bibitem{MR3592680}
Beno\^{i}t M\'{e}sognon-Gireau.
\newblock The {C}auchy problem on large time for the water waves equations with large topography variations.
\newblock {\em Ann. Inst. H. Poincar\'{e} C Anal. Non Lin\'{e}aire}, 34(1):89--118, 2017.

\bibitem{MR2418072}
Guy M\'{e}tivier.
\newblock {\em Para-differential calculus and applications to the {C}auchy problem for nonlinear systems}, volume~5 of {\em Centro di Ricerca Matematica Ennio De Giorgi (CRM) Series}.
\newblock Edizioni della Normale, Pisa, 2008.

\bibitem{MR2558419}
Mei Ming and Zhifei Zhang.
\newblock Well-posedness of the water-wave problem with surface tension.
\newblock {\em J. Math. Pures Appl. (9)}, 92(5):429--455, 2009.

\bibitem{MR0609882}
V.~I. Nalimov.
\newblock The {C}auchy-{P}oisson problem.
\newblock {\em Dinamika Splo\v{s}n. Sredy}, 1(18):104--210, 254, 1974.

\bibitem{MR3724757}
Huy~Quang Nguyen.
\newblock A sharp {C}auchy theory for the 2{D} gravity-capillary waves.
\newblock {\em Ann. Inst. H. Poincar\'{e} C Anal. Non Lin\'{e}aire}, 34(7):1793--1836, 2017.

\bibitem{rowan2023dimensional}
James Rowan and Lizhe Wan.
\newblock Two dimensional solitary water waves with constant vorticity, part {I}: the deep gravity case, 2023.

\bibitem{rowan2024}
James Rowan and Lizhe Wan.
\newblock Two-dimensional solitary water waves with constant vorticity, part {II}: the deep capillary case, 2024.

\bibitem{MR2763036}
Jalal Shatah and Chongchun Zeng.
\newblock Local well-posedness for fluid interface problems.
\newblock {\em Arch. Ration. Mech. Anal.}, 199(2):653--705, 2011.

\bibitem{MR2262949}
Erik Wahl\'en.
\newblock Steady periodic capillary-gravity waves with vorticity.
\newblock {\em SIAM J. Math. Anal.}, 38(3):921--943, 2006.

\bibitem{MR2309783}
Erik Wahl\'{e}n.
\newblock A {H}amiltonian formulation of water waves with constant vorticity.
\newblock {\em Lett. Math. Phys.}, 79(3):303--315, 2007.

\bibitem{wan2023low}
Lizhe Wan.
\newblock Low regularity well-posedness for two dimensional deep gravity water waves with constant vorticity, 2023.

\bibitem{MR3585049}
Chao Wang and ZhiFei Zhang.
\newblock Break-down criterion for the water-wave equation.
\newblock {\em Sci. China Math.}, 60(1):21--58, 2017.

\bibitem{MR0660822}
Hideaki Yosihara.
\newblock Gravity waves on the free surface of an incompressible perfect fluid of finite depth.
\newblock {\em Publ. Res. Inst. Math. Sci.}, 18(1):49--96, 1982.

\bibitem{MR0728155}
Hideaki Yosihara.
\newblock Capillary-gravity waves for an incompressible ideal fluid.
\newblock {\em J. Math. Kyoto Univ.}, 23(4):649--694, 1983.

\bibitem{zakharov1968stability}
A.~V. Zakharov.
\newblock Stability of periodic solutions of a second-order differential equation with delay.
\newblock {\em Journal of Applied Mechanics and Technical Physics}, 9(2):190--194, 1968.

\bibitem{MR4072685}
Hui Zhu.
\newblock Control of three dimensional water waves.
\newblock {\em Arch. Ration. Mech. Anal.}, 236(2):893--966, 2020.

\end{thebibliography}
\bibliographystyle{plain}
\end{document}